\documentclass[oneside, a4paper, 10pt]{article}
\usepackage[colorlinks = true,
            linkcolor = black,
            urlcolor  = blue,
            citecolor = black]{hyperref}
\usepackage{amsmath, amsthm, amssymb, amsfonts, amscd}
\usepackage[left=2cm,right=2cm,top=3cm,bottom=3cm,a4paper]{geometry}
\usepackage{mathtools}
\usepackage{thmtools}
\usepackage{graphicx}
\usepackage{cleveref}
\usepackage{mathrsfs}
\usepackage{titling}
\usepackage{url}
\usepackage[nosort, noadjust]{cite}
\usepackage{enumitem}
\usepackage{tikz-cd}
\usepackage{multirow}
\usepackage{fancyhdr}
\usepackage{sectsty}
\usepackage{longtable}
\usepackage{authblk}
\DeclareGraphicsExtensions{.pdf,.png,.jpg}

\newcommand{\re}{\operatorname{Re}}
\newcommand{\Gal}{\operatorname{Gal}}

\newcommand{\Aut}{\operatorname{Aut}}
\newcommand{\Hom}{\operatorname{Hom}}
\newcommand{\im}{\operatorname{im}}
\newcommand{\GL}{\operatorname{GL}}
\newcommand{\ind}{{\operatorname{ind}}}
\newcommand{\rank}{\operatorname{rank}}
\newcommand{\res}{\operatorname{R}}
\newcommand{\nm}{\operatorname{N}}
\newcommand{\NF}{\operatorname{NF}}
\newcommand{\tor}{\operatorname{tor}}
\newcommand{\cond}{\operatorname{Cond}}
\newcommand{\lcm}{\operatorname{lcm}}

\newcommand{\Sqf}{\operatorname{Sqf}}

\newcommand{\Z}{\mathbb{Z}}
\newcommand{\Q}{\mathbb{Q}}
\newcommand{\R}{\mathbb{R}}

\newcommand{\G}{\mathbb{G}}
\newcommand{\CC}{\mathcal{C}}
\newcommand{\WW}{\mathcal{W}}

\theoremstyle{definition}
\newtheorem{theorem}{Theorem}[section]
\newtheorem{proposition}[theorem]{Proposition}

\newtheorem{lemma}[theorem]{Lemma}
\newtheorem{conjecture}[theorem]{Conjecture}
\newtheorem*{conjecture*}{Conjecture}
\newtheorem{remark}[theorem]{Remark}

\makeatletter 
\newcommand\semilarge{\@setfontsize\semilarge{11}{13.2}}
\makeatother

\title{\semilarge{\textbf{COUNTING $3$-DIMENSIONAL ALGEBRAIC TORI OVER $\Q$}}}
\author{\normalsize{JUNGIN LEE}}
\date{}

\sectionfont{\large}
\subsectionfont{\normalsize}

\renewenvironment{abstract}
 {\quotation\small\noindent\rule{\linewidth}{.5pt}\par\smallskip
  {\centering\bfseries\abstractname\par}\medskip}
 {\par\noindent\rule{\linewidth}{.5pt}\endquotation}

 \AtEndDocument{%
  \bigskip
  \footnotesize

  \textsc{Jungin Lee, Center for Mathematical Challenges, Korea Institute for Advanced Study, Seoul 02455, Korea}\par\nopagebreak
  \textit{E-mail address:} \texttt{jilee.math@gmail.com} }

\begin{document}
\maketitle
\vspace{-18mm}

\begin{abstract}
In this paper we count the number $N_3^{\tor}(X)$ of $3$-dimensional algebraic tori over $\Q$ whose Artin conductor is bounded above by $X$. We prove that $N_3^{\tor}(X) \ll_{\varepsilon} X^{1 + \frac{\log 2 + \varepsilon}{\log \log X}}$, and this upper bound can be improved to $N_3^{\tor}(X) \ll X (\log X)^4 \log \log X$ under the Cohen-Lenstra heuristics for $p=3$. We also prove that for $67$ out of $72$ conjugacy classes of finite nontrivial subgroups of $\GL_3(\Z)$, Malle's conjecture for tori over $\Q$ holds up to a bounded power of $\log X$ under the Cohen-Lenstra heuristics for $p=3$ and Malle's conjecture for quartic $A_4$-fields. 

\end{abstract}

\section{Introduction} \label{Sec1}

Throughout the paper, we assume that every number field is contained in a fixed algebraic closure $\overline{\Q}$ of $\Q$. Let $K$ be a number field and $D_K$ be the absolute value of the discriminant of $K$. By Hermite-Minkowski theorem, there are only finitely many number fields $K$ such that $D_K$ is bounded above by a given number. For an integer $n \geq 2$, let $N_n(X)$ be the number of degree $n$ number fields $K$ such that $D_K \leq X$. For $n \geq 2$ and a transitive subgroup $G \leq S_n$, let $N_n(X; G)$ be the number of degree $n$ number fields $K$ such that $D_K \leq X$ and the Galois group $\Gal(K^c/\Q)$ is permutation-isomorphic to $G$ in $S_n$. Here $K^c$ denotes the Galois closure of $K$ over $\Q$. 

Counting number fields by discriminant (i.e. the asymptotics of the numbers $N_n(X)$ and $N_n(X; G)$) is one of the central problems in arithmetic statistics. Linnik's conjecture states that for each $n \geq 2$, there exists a constant $c_n>0$ such that $N_n(X) \sim c_n X$ as $X \rightarrow \infty$. This conjecture has been proved only for $n \leq 5$. The case $n=2$ is easy, the case $n=3$ was proved by Davenport-Heilbronn \cite{DH71} and the cases $n=4, 5$ were proved by Bhargava \cite{Bha05, Bha10}. For general $n$, upper bound of $N_n(X)$ has been studied by Schmidt \cite{Sch95}, Ellenberg-Venkatesh \cite{EV06}, Couveignes \cite{Cou20} and Lemke Oliver-Thorne \cite{LT22}.

Malle's conjecture \cite{Mal04} states that 
\begin{equation} \label{eq1a}
N_n(X; G) \sim c_G X^{\frac{1}{a(G)}} (\log X)^{b(G)-1}
\end{equation}
for some positive integers $a(G), b(G)$ and a constant $c_G>0$. The index of $g \in G \leq S_n$ is define by 
$$
\ind(g) := n - \text{the number of orbits of } g \text{ on } \left \{ 1, 2, \cdots, n \right \}
$$
and let $a(G) := \min_{g \in G \setminus \left \{ 1 \right \}} \ind(g)$. The number $b(G)$ is defined to be the number of orbits $\CC$ of $\Gal(\overline{\Q}/\Q)$-action on the conjugacy classes of $G$ (via the cyclotomic character) such that the index of some (equivalently, all) $g \in \CC$ is $a(G)$. We refer \cite[Section 1.1]{Lee21} for a summary on the known results on Malle's conjecture. 

As a natural generalization of counting number fields by discriminant, the author \cite{Lee21} studied counting (algebraic) tori over $\Q$ by Artin conductor. Let $T$ be an $n$-dimensional tori over $\Q$ with a splitting field $L$ and $X^*(T) := \Hom_{\overline{\Q}} (T_{\overline{\Q}}, \G_{m, \overline{\Q}})$ be its character group. Then the Galois group $\Gal(\overline{\Q} / \Q)$ acts on $X^*(T)$ by conjugation, and this induces the representation
$$
\rho_T : \Gal(\overline{\Q} / \Q) \rightarrow \Aut (X^*(T)) \cong \GL_n(\Z).
$$
Its image $G_T := \im(\rho_T)$ is a finite subgroup of $\GL_n(\Z)$ isomorphic to $\Gal(L/\Q)$. Since the isomorphism $\Aut (X^*(T)) \cong \GL_n(\Z)$ depends on the choice of the $\Z$-basis of $X^*(T)$, $\rho_T$ and $G_T$ are well-defined only up to conjugation. Let $C(T)$ be the Artin conductor of the representation
$$
\rho : \Gal(L/\Q) \rightarrow \Aut (X^*(T)_{\Q}) \cong \GL_n(\Q)
$$
induced by $\rho_T$. For a degree $n$ number field $K$, $T = \res_{K/\Q} \G_m$ (Weil restriction of $\G_m$) is an $n$-dimensional torus over $\Q$ whose splitting field is $K^c$ and the Artin conductor is $C(\res_{K/\Q} \G_m)=D_K$ (see \cite[Section 1.2]{Lee21}). This shows that counting tori over $\Q$ of given dimension by Artin conductor is a generalization of counting number fields of given degree by discriminant. 

Let $N_n^{\tor}(X)$ be the number of the isomorphism classes of tori over $\Q$ of dimension $n$ such that $C(T) \leq X$. For a finite subgroup $H \neq 1$ of $\GL_n(\Z)$, $N_n^{\tor}(X; H)$ denotes the number of such tori $T$ over $\Q$ such that $G_T$ is conjugate to $H$ in $\GL_n(\Z)$. The following two conjectures from \cite{Lee21} are analogues of Linnik's and Malle's conjectures for tori over $\Q$. Note that the second conjecture follows from a more general conjecture of Ellenberg and Venkatesh \cite[Question 4.3]{EV05}, so it is not new. It is also remarkable that the second conjecture implies the first conjecture \cite[Corollary 3.6]{Lee21}.

\begin{conjecture} \label{conj1a}
(\cite[Conjecture 3.1]{Lee21}) For every $n \geq 1$, there exists a constant $c_n>0$ satisfying
\begin{equation}
N_n^{\tor}(X) \sim c_n X (\log X)^{n-1}.
\end{equation}
\end{conjecture}

\begin{conjecture} \label{conj1b}
(\cite[Conjecture 3.2]{Lee21}) For every $n \geq 1$ and a finite subgroup $1 \neq H \leq \GL_n(\Z)$, 
\begin{equation}
N_n^{\tor}(X; H) \sim c_H X^{\frac{1}{a(H)}} (\log X)^{b(H)-1}
\end{equation}
where the positive integers $a(H), b(H)$ and a constant $c_H>0$ depend only on $H$. For an $n \times n$ identity matrix $I_n$, the number $a(H)$ is given by
\begin{equation}
a(H) := \min_{h \in H \setminus \left \{ I_n \right \}} \rank(h - I_n)
\end{equation}
and the number $b(H)$ is given by the number of the orbits $\CC$ of the action of $\Gal(\overline{\Q} / \Q)$ on the conjugacy classes of $H$ via the cyclotomic character such that $\rank(h - I_n) = a(H)$ for some (equivalently, all) $h \in \CC$. 
\end{conjecture}

The above conjectures are trivial for the case $n=1$. For the $2$-dimensional case, the following results are known. Here $H_{12, A}$ is a finite subgroup of $\GL_2(\Z)$ which is isomorphic to the dihedral group $D_6$ of order $12$. Such group $H_{12, A}$ is unique up to conjugation and satisfies $a(H_{12, A})=1$ and $b(H_{12, A})=2$. We also use the asymptotic notation $f(X) \ll_{\varepsilon} g(X, \varepsilon)$, which means that $f(X) \ll g(X, \varepsilon)$ for every $\varepsilon > 0$.

\begin{proposition} \label{prop1c}
\begin{enumerate}
    \item (\cite[Proposition 4.1]{Lee21}) Conjecture \ref{conj1b} holds for every finite nontrivial subgroup of $\GL_2(\Z)$ which is not conjugate to $H_{12, A}$.

    \item (\cite[Theorem 4.9]{Lee21}) We have the followings:
    \begin{subequations} 
\begin{align}
X \ll N_2^{\tor}(X; H_{12, A}) & \ll_{\varepsilon} X^{1+\frac{\log 2 + \varepsilon}{\log \log X}} \\
X \log X \ll N_2^{\tor}(X) & \ll_{\varepsilon} X^{1+\frac{\log 2 + \varepsilon}{\log \log X}}. 
\end{align}
\end{subequations}
    
    \item (\cite[Theorem 4.10]{Lee21}) Under the assumption of Conjecture \ref{conj1d}, we have
\begin{equation} 
N_2^{\tor}(X; H_{12, A}) \leq N_2^{\tor}(X) 
\ll_{\varepsilon} X (\log X)^{1 + \varepsilon}.
\end{equation}
\end{enumerate}
\end{proposition}

In the above proposition, the following version of the Cohen-Lenstra heuristics was used. Denote by $\NF_2^{+}$ (resp. $\NF_2^{-}$) the set of all real (resp. imaginary) quadratic fields whose elements are ordered by the absolute values of their discriminants. For a number field $K$ and a prime $p$, denote the size of the $p$-torsion subgroup of the class group of $K$ by $h_p(K)$. For an odd prime $p$ and a positive integer $\alpha$, consider the following version of the Cohen-Lenstra heuristics.

\begin{itemize}
\item (\cite[(C10)]{CL84}) $\text{Conj}^+(p, \alpha)$ : The average of $\displaystyle \prod_{0 \leq i < \alpha} (h_p(K)-p^i)$ for $K \in \NF_2^{+}$ is $p^{-\alpha}$.

\item (\cite[(C6)]{CL84}) $\text{Conj}^-(p, \alpha)$ : The average of $\displaystyle \prod_{0 \leq i < \alpha} (h_p(K)-p^i)$ for $K \in \NF_2^{-}$ is $1$.
\end{itemize} 

The above conjectures are true for $p=3$ and $\alpha = 1$ by \cite[Theorem 3]{DH71}, but they are still open for the other cases. In many cases the assumption of the conjectures for $p=3$ and every $\alpha > 0$ improves the upper bounds for counting tori over $\Q$.

\begin{conjecture} \label{conj1d}
$\text{Conj}^+(3, \alpha)$ and $\text{Conj}^-(3, \alpha)$ are true for every positive integer $\alpha$. 
\end{conjecture}

The purpose of this paper is to count the number of $3$-dimensional tori over $\Q$ by Artin conductor. First we classify the $3$-dimensional tori over $\Q$ and compute their Artin conductors in Section \ref{Sec3}. The classification is much more complicated compared to the $2$-dimensional case. After that, we estimate the magnitude of $N_3^{\tor}(X; H)$ for each finite nontrivial subgroup $H$ of $\GL_3(\Z)$ in Section \ref{Sec4}. We do this for abelian $H$ in Section \ref{Sub41} and non-abelian $H$ in Section \ref{Sub42} and \ref{Sub43}. The results can be summarized as follow. 

\begin{theorem} \label{thm1e}
Let $H$ be a finite nontrivial subgroup of $\GL_3(\Z)$.
\begin{enumerate}
    \item (Proposition \ref{prop41a}, \ref{prop41c}) Conjecture \ref{conj1b} is true for every abelian $H$.
    
    \item (Theorem \ref{thm42e}) $X^{\frac{1}{a(H)}} \ll N_3^{\tor}(X; H)$ for every $H$.
    
    \item (Proposition \ref{prop41a}, \ref{prop41c}, \ref{prop42a}, \ref{prop43a}, \ref{prop43b}) Under the assumption of Conjecture \ref{conj1d} and Malle's conjecture for quartic $A_4$-fields, we have
    $$
    N_3^{\tor}(X; H) \ll_{\varepsilon} X^{\frac{1}{a(H)}} (\log X)^{6 + \varepsilon}
    $$
    for $67$ out of $72$ conjugacy classes of finite nontrivial subgroups of $\GL_3(\Z)$. 
\end{enumerate}
\end{theorem}
It is notable that the ratios of the upper and lower bounds of $N_3^{\tor}(X; H)$ are $(\log X)^{O(1)}$ for most of the finite subgroups $H$ of $\GL_3(\Z)$. Since we have
$$
N_3^{\tor}(X) = \sum_{H} N_3^{\tor}(X; H)
$$
where $H$ runs through the conjugacy classes of finite nontrivial subgroups of $\GL_3(\Z)$, we can bound the size of $N_3^{\tor}(X)$. The following theorem on the upper bound of $N_3^{\tor}(X)$ is the main result of the paper.

\begin{theorem} \label{thm1f}
(Theorem \ref{thm43g})
\begin{enumerate}
    \item We have
    \begin{equation} \label{eq1f1}
    N_3^{\tor}(X) \ll_{\varepsilon} X^{1 + \frac{\log 2 + \varepsilon}{\log \log X}}.
    \end{equation}
    
    \item Under the assumption of Conjecture \ref{conj1d}, we have
    \begin{equation} \label{eq1f2}
    N_3^{\tor}(X) \ll X (\log X)^4 \log \log X.
    \end{equation}
\end{enumerate}
\end{theorem}

\section{Preliminaries} \label{Sec2}

\subsection{Known results on counting number fields} \label{Sub21}

In this section, $C$ denotes a positive constant which may change from line to line. First we give a list of known cases of Malle's conjecture which will be used in the sequel. 

\begin{proposition} \label{prop21a}
The equation (\ref{eq1a}) holds for the following cases: 
\begin{enumerate}
    \item (\cite{Mak85}) $G \leq S_{\left | G \right |}$, $G$ abelian
    
    \item (\cite{DH71, CDO02, Bha05}) $S_3 \leq S_3$, $D_4 \leq S_4$, $S_4 \leq S_4$ 
    
    \item (\cite[Theorem 1.1]{MTTW20}) $D_6 \cong S_3 \times C_2 \leq S_6$ (i.e. $N_6(X; D_6) \sim CX^{\frac{1}{2}}$)
    
    \item (\cite[Theorem 1]{SV21}) $D_4 \leq S_8$ (i.e. $N_8(X; D_4) \sim CX^{\frac{1}{4}} (\log X)^2$).
\end{enumerate}
\end{proposition}

There also has been some progress on counting number fields by an invariant other than the discriminant. 
For a positive integer $n \geq 2$ and a transitive subgroup $G \leq S_n$, let $\NF_n(G)$ be the set of degree $n$ number fields such that the Galois group $\Gal(K^c/\Q)$ is permutation-isomorphic to $G$. 
Let $I$ be an invariant of number fields such that for every $X>0$, there are finitely many number fields $K$ such that $I(K) \leq X$. 
Denote by $N_n(X; G; I)$ the number of $K \in \NF_n(G)$ such that $I(K) \leq X$. 
Sometimes we write $N_n(X; G; I)$ by $N_{n}^{K}(X; G; I)$ to clarify that $I$ is an invariant of $K$. 
It is clear from the definition that $N_n(X; G; I) = N_n(X; G)$ if $I(K) = D_K$. Denote $\NF_2(C_2)$ (resp. $N_2(X; C_2; I)$) by $\NF_2$ (resp. $N_2(X; I)$) for simplicity.

The next proposition is a collection of results on the asymptotics of $N_n(X; G; I)$ where $I$ is an invariant other than the discriminant. The first two results are direct consequences of the work of M\"aki \cite{Mak93} on the asymptotics of the number of abelian number fields with bounded conductor (cf. \cite[Section 2.3]{Lee21}).

\begin{proposition} \label{prop21b}
\begin{enumerate}
    \item For $L_4 \in \NF_4(C_4)$, denote its unique quadratic subfield by $L_2$. Then
    \begin{equation} \label{eq21b}
N_4(X; C_4; \frac{D_{L_4}}{D_{L_2}}) \sim C X^{\frac{1}{2}} \log X.
\end{equation}

    \item For $L_6 \in \NF_6(C_6)$, denote its unique cubic (resp. quadratic) subfield by $L_3$ (resp. $L_2$). Then
    \begin{equation} \label{eq21c}
N_6(X; C_6; \frac{D_{L_6}}{D_{L_2}D_{L_3}}) \sim C X^{\frac{1}{2}} (\log X)^2.
\end{equation}
    
    \item (\cite[Theorem 1]{ASVW21}) For $L_4 \in \NF_4(D_4)$, denote its unique quadratic subfield by $L_2$. Then
    \begin{equation} \label{eq21d}
    N_4(X; D_4; \frac{D_{L_4}}{D_{L_2}}) = C X \log X + O(X \log \log X).
    \end{equation}
\end{enumerate}
\end{proposition}

For $L_6 \in \NF_6(D_6)$, denote its unique cubic (resp. quadratic) subfield by $L_3$ (resp. $L_2$). Then
$$
N_2^{\tor}(X; H_{12, A}) = N_6(X; D_6; \frac{D_{L_6}}{D_{L_3}D_{L_2}})
$$
(cf. \cite[Section 4.1]{Lee21}) so Proposition \ref{prop1c} implies that
\begin{equation} \label{eq21e}
N_6(X; D_6; \frac{D_{L_6}}{D_{L_3}D_{L_2}})
\ll_{\varepsilon} X^{1+\frac{\log 2 + \varepsilon}{\log \log X}}
\ll_{\varepsilon} X^{1+\varepsilon}
\end{equation}
and this can be improved to
\begin{equation} \label{eq21f}
N_6(X; D_6; \frac{D_{L_6}}{D_{L_3}D_{L_2}})
\ll_{\varepsilon} X (\log X)^{1+\varepsilon}
\end{equation}
under the assumption of Conjecture \ref{conj1d}.

We also introduce a proposition which concerns the product distribution appears in counting number fields. It is useful when we consider the compositum of two linearly disjoint number fields.

\begin{proposition} \label{prop21c}
Let $F_i(X) = \# \left \{ s \in S_i : s \leq X  \right \}$ ($i= 1, 2$) be the asymptotic distribution of some multi-set $S_i$ consists of a sequence of elements of $\R_{\geq 1}$. Suppose that $F_i(X) \sim A_i X^{n_i} (\log X)^{r_i}$ for $A_i >0$, $n_i>0$ and $r_i \in \Z_{\geq 0}$. Consider the product distribution
$$
P(X) := \# \left \{ (s_1, s_2) \in S_1 \times S_2 : s_1s_2 \leq X  \right \}.
$$
\begin{enumerate}
    \item (\cite[Lemma 3.1]{Wan21}) If $n_1=n_2=n$, then 
    $$
    P(X) \sim A_1A_2 \frac{r_1 ! r_2 !}{(r_1+r_2+1)!}nX^n (\log X)^{r_1+r_2+1}.
    $$
    \item (\cite[Lemma 3.2]{Wan21}) If $n_1 > n_2$, then there exists a constant $C>0$ such that
    $$
    P(X) \sim C X^{n_1} (\log X)^{r_1}.
    $$
\end{enumerate}
\end{proposition}

\subsection{Discriminants of number fields} \label{Sub22}

For some Galois extensions of number fields, there are algebraic relations between the discriminants of their subextensions. We provide such relations for Galois extensions whose Galois group is isomorphic to one of the groups $C_2^2$, $S_3$, $A_4$ and $S_4$. The formulas in the following proposition will be used frequently in Section \ref{Sec3}. 
\begin{proposition} \label{prop22a}
Let $L/K$ be a Galois extension of number fields. 
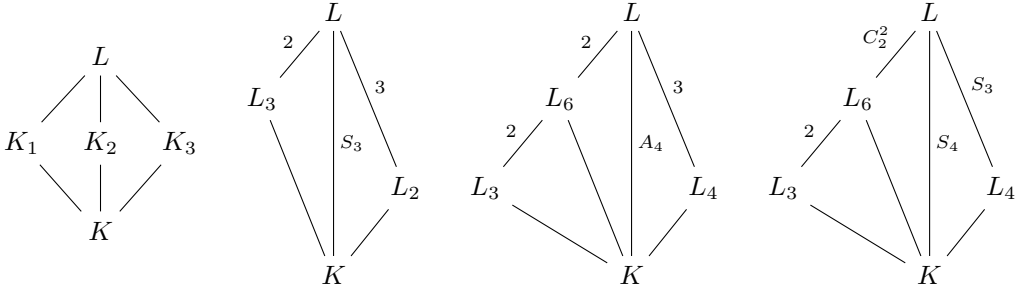
\[ 
\begin{tikzcd}[column sep=0.8em]
& L \arrow[ld, no head] \arrow[d, no head] \arrow[rd, no head] & \\
K_1 \arrow[rd, no head] & K_2 \arrow[d, no head]  & K_3 \arrow[ld, no head] \\
& K  & 
\end{tikzcd}
\quad
\begin{tikzcd}[column sep=0.8em]
& L \arrow[ddd, "S_3", no head] \arrow[rdd, "3", no head] & \\
L_3 \arrow[ru, "2", no head] \arrow[rdd, no head] &   & \\
& & L_2 \\
& K \arrow[ru, no head]  & 
\end{tikzcd}
\quad
\begin{tikzcd}[column sep=0.8em]
&  & L \arrow[ld, "2"', no head] \arrow[rdd, "3", no head] \arrow[ddd, "A_4", no head] & \\
& L_6 \arrow[ld, "2"', no head] \arrow[rdd, no head] &  & \\ 
L_3 \arrow[rrd, no head] &   &   & L_4 \arrow[ld, no head] \\
&  & K  &   
\end{tikzcd}
\quad
\begin{tikzcd}[column sep=0.8em]
&  & L \arrow[ld, "C_2^2"', no head] \arrow[rdd, "S_3", no head] \arrow[ddd, "S_4", no head] & \\
& L_6 \arrow[ld, "2"', no head] \arrow[rdd, no head] &  & \\
L_3 \arrow[rrd, no head] &   &   & L_4 \arrow[ld, no head] \\
&  & K  &   
\end{tikzcd}
\]

\begin{enumerate}
    \item Assume that $\Gal(L/K) \cong C_2^2$ and let $K_1$, $K_2$ and $K_3$ be the quadratic subextensions of $L/K$. Then we have
    \begin{equation} \label{eq22a}
    D_L D_K^2 = D_{K_1}D_{K_2}D_{K_3}.
    \end{equation}
    
    \item Assume that $\Gal(L/K) \cong S_3$, let $L_3$ be one of the cubic subextensions of $L/K$ and $L_2$ be the unique quadratic subextension of $L/K$. Then we have
    \begin{equation} \label{eq22b}
    D_L D_K^2 = D_{L_3}^2D_{L_2}.
    \end{equation}
    
    \item Assume that $\Gal(L/K) \cong A_4$, let $L_6$ be one of the sextic subextensions of $L/K$, $L_4$ be one of the quartic subextensions of $L/K$ and $L_3$ be the unique cubic subextension of $L/K$. (In this case $L_3$ is a subfield of $L_6$.) Then we have the following formulas: 
    \begin{subequations} 
\begin{align}
D_L D_K^3 &= D_{L_4}^3D_{L_3} \label{eq22c1} \\
   D_{L_6}D_K &= D_{L_4}D_{L_3} \label{eq22c2} \\
   D_LD_K^2 &= D_{L_6}D_{L_4}^2. \label{eq22c3}
  \end{align}
\end{subequations}

    \item Assume that $\Gal(L/K) \cong S_4$, let $L_i$ ($i=3, 4$) be one of the degree $i$ subextensions of $L/K$ and $L_6$ be the unique subextension of $L/K$ such that $L_3 \subset L_6$, $\Gal(L/L_6) \cong C_2^2$ and the Galois closure of $L_6/K$ is $L$. Then we have
    \begin{equation} \label{eq22d}
    D_{L_6}D_K = D_{L_4}D_{L_3}.
    \end{equation}
\end{enumerate}
\end{proposition}

\begin{proof}
\begin{enumerate}
    \item By the conductor-discriminant formula \cite[VII.11.9]{Neu99}, we have $D_{L/K}=D_{K_1/K}D_{K_2/K}D_{K_3/K}$ where $D_{L/K} := \nm_{K/\Q}(\mathcal{D}_{L/K})$ is the absolute norm of the relative discriminant $\mathcal{D}_{L/K}$.
    
    \item $S_3=C_3 \rtimes C_2$ is a Frobenius group so it is a consequence of \cite[Theorem 4]{FK03}.
    
    \item $A_4=C_2^2 \rtimes C_3$ is a Frobenius group so the first formula is a consequence of \cite[Theorem 4]{FK03}. For a number field $M$, denote its Dedekind zeta function by $\zeta_M(s)$ and denote the number of real and complex embeddings of $M$ by $r_M$ and $s_M$, respectively. Then the character theory of $A_4$ implies the relation
    \begin{equation} \label{eq22e}
    \zeta_{L_6}(s)\zeta_{K}(s) = \zeta_{L_4}(s)\zeta_{L_3}(s)
    \end{equation}
    (cf. \cite[Theorem 5.1]{CT16}) and the functional equation of the Dedekind zeta function \cite[Corollary VII.5.10]{Neu99} implies that
    \begin{equation} \label{eq22f}
    D_M^{s-\frac{1}{2}} \cdot \frac{\zeta_{M}(s)}{\zeta_{M}(1-s)} 
\cdot \left ( \frac{ \Gamma(\frac{s}{2})}{\Gamma(\frac{1-s}{2})} \right )^{r_M}
\cdot \left ( \frac{2^{1-2s} \Gamma(s)}{\Gamma(1-s)} \right )^{s_M}
\cdot \pi^{(\frac{1-2s}{2} [M:\Q])}=1
    \end{equation}
    for any number field $M$.

    Now denote $c := s_K + s_{L_6} - s_{L_3} - s_{L_4} \in \Z$. The equations (\ref{eq22e}) and (\ref{eq22f}) imply that
    $$
    \left ( \frac{D_K D_{L_6}}{D_{L_4}D_{L_3}} \right )^{s-\frac{1}{2}} = \left ( \frac{2^{1-2s} \Gamma(s) \Gamma(\frac{1-s}{2})^2}{\Gamma(1-s) \Gamma(\frac{s}{2})^2} \right )^c.
    $$
    Substituting $s = \frac{3}{2}$ into the above equation, we obtain the equation
    $$
    \frac{D_K D_{L_6}}{D_{L_4}D_{L_3}} 
    = \left (  \frac{\frac{1}{4} \Gamma(\frac{3}{2}) \Gamma(-\frac{1}{4})^2}{\Gamma(-\frac{1}{2}) \Gamma(\frac{3}{4})^2}  \right )^c
    = (-1)^c
    $$
    and its left-hand side is positive. This proves the second formula, and the last formula comes from the first and second formulas.

    \item It can be proved as in the proof of the formula (\ref{eq22c2}), except that the character theory of $A_4$ is replaced by the character theory of $S_4$ (cf. \cite[Theorem 5.1]{CT16}). \qedhere
\end{enumerate}
\end{proof}

For a prime $p$ and a positive integer $m$, denote the exponent of $p$ in $m$ by $v_p(m)$. If $M$ is a number field which is tamely ramified at $p$, the inertia group $I_{M, p}$ is cyclic so we can choose its generator $g_{M, p}$. The following proposition gives a description of the discriminant of compositum of two number fields. See \cite[Section 2]{Wan21} for details.

\begin{proposition} \label{prop22b}
(\cite[Theorem 2.2 and 2.3]{Wan21}) Let $K_1$ and $K_2$ be number fields such that $K_1^c \cap K_2^c = \Q$ and $p$ be a prime such that both of $K_1$ and $K_2$ are tamely ramified at $p$. Suppose that $g_{K_1, p} = \prod_{k} c_k$ (product of disjoint cycles) and $g_{K_2, p}=\prod_{l} d_l$. Then
$$
v_p(D_{K_1K_2}) = m_1m_2 - \sum_{k, l} \gcd ( \left | c_{k} \right |,  \left | d_l \right | ),
$$
where $m_i$ is the degree of $K_i$ and $\left | c \right |$ denotes the length of the cycle $c$. If the least common multiple of $\left | c_{k} \right |$ and the least common multiple of $\left | d_l \right |$ are coprime, then we have
$$
v_p(D_{K_1K_2}) 
=v_p(D_{K_1}) \cdot m_2 + v_p(D_{K_2}) \cdot m_1 - v_p(D_{K_1})v_p(D_{K_2}).
$$
\end{proposition}

\subsection{Analytic preliminaries} \label{Sub23}

The following version of the Tauberian theorem is very useful for counting number fields. For instance, Wright \cite{Wri89} proved Malle's conjecture for abelian extensions by studying their associated Dirichlet series and applying the Tauberian theorem. 

\begin{proposition} \label{prop23a}
(\cite[Theorem III]{Del54}) Let $\displaystyle f(s)=\sum_{n=1}^{\infty} \frac{a_n}{n^s}$ ($a_n \geq 0$) be a Dirichlet series which converges for $\re(s) > a > 0$. Assume that
$$
f(s) = \frac{g(s)}{(s-a)^w} + h(s)
$$
in the domain of convergence, where $g(s)$ and $h(s)$ are holomorphic in $\re(s) \geq a$, $g(a) \neq 0$ and $w$ is a positive integer. Then we have
$$
\sum_{n \leq X} a_n \sim \frac{g(a)}{a \Gamma(w)} X^a (\log X)^{w-1},
$$
where $\Gamma(w)$ denotes the Gamma function.
\end{proposition}

For $\alpha > \beta > 0$, denote $f \in \mathcal{M} (\alpha, \beta)$ if $f(s)$ converges for $\re (s) > \alpha$ and has a meromorphic continuation to $\re (s) > \beta$ which is holomorphic except for a simple pole at $s = \alpha$. The next lemma will be used in the proof of Proposition \ref{prop42c}. 

\begin{lemma} \label{lem23b}
We have
$$
g(s) := \prod_{p \equiv \pm 1 \,\, (\text{mod } 7) } \left ( 1+\frac{3}{p^s} \right ) \in \mathcal{M} (1, \frac{1}{2}).
$$
\end{lemma}

\begin{proof}
For a positive integer $q>1$ and an integer $a$ prime to $q$, define the Dirichlet series
$$
B_{q, a}(s) := \prod_{p \equiv a \,\, ( \text{mod } q)} (1-\frac{1}{p^s})^{-1}.
$$
Let $\chi$ be a Dirichlet character modulo $q$ and $\chi_0$ be the principal character modulo $q$. By the answer of Johan Andersson to the MathOverflow question 28000 \cite{And11}, we have
$$
B_{q, a}(s) = L(s, \chi_0)^{\frac{1}{\varphi (q)}} \prod_{\chi \neq \chi_0} L(s, \chi)^{\frac{\overline{\chi(a)}}{\varphi (q)}} A_{q, a}(s)
$$
where $A_{q, a}(s)$ is holomorphic and non-vanishing for $\displaystyle \re (s) > \frac{1}{2}$. 

A direct computation shows that
$$
g(s) = B_{7, 1}(s)^3 B_{7, -1}(s)^3 h(s)
$$ 
for
$$
h(s) := \prod_{p \equiv \pm 1 \,\, (\text{mod } 7) } \left ( 1 - \frac{6}{p^{2s}} + \frac{8}{p^{3s}} - \frac{3}{p^{4s}} \right ), 
$$
which is holomorphic and non-vanishing in $\displaystyle \re (s) > \frac{1}{2}$. Therefore it is enough to show that
$$
B_{7, 1}(s)^3 B_{7, -1}(s)^3
= L(s, \chi_0)
\prod_{\chi \neq \chi_0} L(s, \chi)^{\frac{\overline{\chi(1)} + \overline{\chi(-1)}}{2}} A_{7, 1}(s)^3 A_{7, -1}(s)^3 \in \mathcal{M} (1, \frac{1}{2}).
$$
Since $A_{7, 1}(s)^3 A_{7, -1}(s)^3$ is holomorphic and non-vanishing for $\displaystyle \re (s) > \frac{1}{2}$ and $\displaystyle \frac{\overline{\chi(1)} + \overline{\chi(-1)}}{2} \in \left \{ 0, 1 \right \}$, the following well-known properties of Dirichlet $L$-functions finish the proof. 
\begin{itemize}
    \item If $\chi \neq \chi_0$, then $L(s, \chi)$ extends to a holomorphic function on $\re (s) >0$ and $L(1, \chi) \neq 0$. 
    
    \item If $\chi = \chi_0$, then $L(s, \chi_0)$ extends to a meromorphic function on $\re (s) >0$ which is holomorphic except for a simple pole at $s=1$. \qedhere
\end{itemize}
\end{proof}

The next lemma follows directly from Wirsing's theorem \cite[Satz 1]{Wir61}. For a positive integer $n$, denote by $\tau(n)$ the number of positive divisors of $n$.

\begin{lemma} \label{lem23c}
(\cite[Lemma 2.1]{BP06}) For any positive real number $t$, we have
$$
\sum_{n \leq X} \tau(n)^t \ll_t X (\log X)^{2^t-1}.
$$
\end{lemma}

Even though the lemma is given in \cite{BP06} only for the case where $t$ is an integer, its proof works for every positive real number $t$.

\section{Classification of $3$-dimensional tori over $\Q$} \label{Sec3}

In this section, we provide a classification of the isomorphism classes of $3$-dimensional tori over $\Q$, together with their Artin conductors. 
There are $73$ conjugacy classes of finite subgroups of $\GL_3(\Z)$, which are computed by Tahara \cite{Tah71} (and corrected by Ascher and Grimmer \cite{AG72}). 
For each finite subgroup $H \neq 1$ of $\GL_3(\Z)$, we compute tori $T$ over $\Q$ such that $G_T$ is conjugate to $H$ in $\GL_3(\Z)$. The computations of $T$ can be done as in \cite[Example 2.1]{Lee21} so we omit them here. 
The next proposition, which summarizes \cite[Section 2.2]{Lee21}, will be frequently used for the computation of the Artin conductor $C(T)$.

\begin{proposition} \label{prop3a}
\begin{enumerate}
    \item For a number field $K$, $C(\res_{K/\Q} \G_m) = D_K$. 
    
    \item If $1 \rightarrow T_1 \rightarrow T_2 \rightarrow T_3 \rightarrow 1$ is an exact sequence of tori over $\Q$, then $C(T_2)=C(T_1)C(T_3)$. 
    
    \item For an extension $L/K$ of number fields, 
    $$
    T_{L/K} := \ker(\res_{L/\Q} \G_m \xrightarrow{\nm_{L/K}} \res_{K/\Q} \G_m)
    $$
    ($\nm_{L/K}$ denotes the norm map) is a torus over $\Q$ of dimension $[L:\Q]-[K:\Q]$ and 
    \begin{equation} \label{eq3a}
    C(T_{L/K}) = \frac{D_L}{D_K}.
\end{equation}

    \item Suppose that $K_1$ and $K_2$ are linearly disjoint number fields and let $L=K_1K_2$. Then $T_{L/K_1} \cap T_{L/K_2}$ is a torus over $\Q$ of dimension $([K_1:\Q]-1)([K_2 : \Q]-1)$ and
    \begin{equation} \label{eq3b}
    C(T_{L/K_1} \cap T_{L/K_2}) = \frac{D_L}{D_{K_1}D_{K_2}}.
\end{equation}
\end{enumerate}
\end{proposition}

In each case, denote the splitting field of a torus $T$ by $L$ and identify $\Gal(L/\Q)$ with $G_T$. Let $C_m$ be the cyclic group of order $m$, $D_m$ be the dihedral group of order $2m$, $S_m$ be the symmetric group of degree $m$ and $A_m$ be the alternating group of degree $m$. For simplicity, denote $D_i := D_{L_i}$, $D_i' := D_{L_i'}$, and so on. We warn the readers not to confuse the discriminant $D_i$ and the dihedral group $D_m$. The following list gives the classification of $3$-dimensional tori over $\Q$ (except for the trivial one $\G_m^3$), together with their Artin conductors.

\begin{enumerate}[label=(\roman*)]

\item $G_T \cong C_2$ : $T$ is one of the following types. 

        \begin{enumerate}[label=(\alph*)]
        \item \label{tor2a} $G_T = H_{2,a} := \left \langle \bigl(\begin{smallmatrix}
1 & 0 & 0\\ 
0 & -1 & 0\\ 
0 & 0 & -1
\end{smallmatrix}\bigr) \right \rangle$ : $T = \G_m \times T_{L/\Q}^2$ and $C(T)=D_L^2$.
        
        \item \label{tor2b} $G_T = H_{2,b} := \left \langle \bigl(\begin{smallmatrix}
-1 & 0 & 0\\ 
0 & 1 & 0\\ 
0 & 0 & 1
\end{smallmatrix}\bigr) \right \rangle$ : $T = \G_m^2 \times T_{L/\Q}$ and $C(T)=D_L$.
        
        \item \label{tor2c} $G_T = H_{2,c} := \left \langle \bigl(\begin{smallmatrix}
-1 & 0 & 0\\ 
0 & 0 & 1\\ 
0 & 1 & 0
\end{smallmatrix}\bigr) \right \rangle$ : $T = \res_{L/\Q}\G_m \times T_{L/\Q}$ and $C(T)=D_L^2$.
        
        \item \label{tor2d} $G_T = H_{2,d} := \left \langle \bigl(\begin{smallmatrix}
1 & 0 & 0\\ 
0 & 0 & -1\\ 
0 & -1 & 0
\end{smallmatrix}\bigr) \right \rangle$ : $T = \G_m \times \res_{L/\Q}\G_m$ and $C(T)=D_L$.
        
        \item \label{tor2e} $G_T = H_{2,e} := \left \langle \bigl(\begin{smallmatrix}
-1 & 0 & 0\\ 
0 & -1 & 0\\ 
0 & 0 & -1
\end{smallmatrix}\bigr) \right \rangle$ : $T = T_{L/\Q}^3$ and $C(T)=D_L^3$.
    \end{enumerate}
    
\item $G_T \cong C_3$ : $T$ is one of the following types. 

    \begin{enumerate}[label=(\alph*)]
    \item \label{tor3a} $G_T = H_{3,a} := \left \langle \bigl(\begin{smallmatrix}
1 & 0 & 0\\ 
0 & 0 & -1\\ 
0 & 1 & -1
\end{smallmatrix}\bigr) \right \rangle$ : $T = \G_m \times T_{L/\Q}$ and $C(T)=D_L$.
        
    \item \label{tor3b} $G_T = H_{3,b} := \left \langle \bigl(\begin{smallmatrix}
0 & 1 & 0\\ 
0 & 0 & 1\\ 
1 & 0 & 0
\end{smallmatrix}\bigr) \right \rangle$ : $T = \res_{L/\Q}\G_m$ and $C(T)=D_L$.
    \end{enumerate}

\item $G_T=\left \langle g \right \rangle \cong C_4$ : $L_4 = L$ has the unique quadratic subfield $L_2 = L^{g^2}$. $T$ is one of the following types. 

\begin{enumerate}[label=(\alph*)]
    \item \label{tor4a} $G_T = H_{4,a} := \left \langle \bigl(\begin{smallmatrix}
1 & 0 & 0\\ 
0 & 0 & -1\\ 
0 & 1 & 0
\end{smallmatrix}\bigr) \right \rangle$ : $T = \G_m \times T_{L_4/L_2}$ and $\displaystyle C(T)=\frac{D_4}{D_2}$.
        
    \item \label{tor4b} $G_T = H_{4,b} := \left \langle \bigl(\begin{smallmatrix}
-1 & 0 & 0\\ 
0 & 0 & 1\\ 
0 & -1 & 0
\end{smallmatrix}\bigr) \right \rangle$ : $T = T_{L_2/\Q} \times T_{L_4/L_2}$ and $C(T)=D_4$.

    \item \label{tor4c} $G_T = H_{4,c} := \left \langle \bigl(\begin{smallmatrix}
1 & 0 & 1\\ 
0 & 0 & -1\\ 
0 & 1 & 0
\end{smallmatrix}\bigr) \right \rangle$ : $T = \left \{ v \in \res_{L_4/\Q} \G_m : v \cdot g^2v = gv \cdot g^3v \right \}$. Since
$$
1 \rightarrow T \rightarrow \res_{L_4/\Q} \G_m
\xrightarrow{v \mapsto \frac{v \cdot g^2v}{gv \cdot g^3v}} T_{L_2/\Q} \rightarrow 1
$$
is exact, we have $\displaystyle C(T)=\frac{D_4}{D_2}$.
        
    \item \label{tor4d} $G_T = H_{4,d} := \left \langle \bigl(\begin{smallmatrix}
-1 & 0 & -1\\ 
0 & 0 & 1\\ 
0 & -1 & 0
\end{smallmatrix}\bigr) \right \rangle$ : $T = T_{L_4/\Q}$ and $C(T)=D_4$.
    \end{enumerate}

\item $G_T=\left \langle g,h \right \rangle \cong C_2 \times C_2$ : $L$ has $3$ quadratic subfields $L_1=L^g$, $L_2=L^h$ and $L_3=L^{gh}$. Then $D_L=D_1D_2D_3$ by the equation (\ref{eq22a}). $T$ is one of the following types. 

\begin{enumerate}[label=(\alph*)]
\addtocounter{enumii}{4}

    \item \label{tor4e} $G_T = H_{4,e} := \left \langle \bigl(\begin{smallmatrix}
1 & 0 & 0\\ 
0 & -1 & 0\\ 
0 & 0 & -1
\end{smallmatrix}\bigr), \bigl(\begin{smallmatrix}
-1 & 0 & 0\\ 
0 & -1 & 0\\ 
0 & 0 & -1
\end{smallmatrix}\bigr) \right \rangle$ : $T = T_{L_1/\Q} \times T_{L_3/\Q}^2$ and $C(T)=D_1D_3^2$.
        
    \item \label{tor4f} $G_T = H_{4,f} := \left \langle \bigl(\begin{smallmatrix}
1 & 0 & 0\\ 
0 & -1 & 0\\ 
0 & 0 & -1
\end{smallmatrix}\bigr), \bigl(\begin{smallmatrix}
-1 & 0 & 0\\ 
0 & -1 & 0\\ 
0 & 0 & 1
\end{smallmatrix}\bigr) \right \rangle$ : $T = T_{L_1/\Q} \times T_{L_2/\Q} \times T_{L_3/\Q}$ and $C(T)=D_L$.

    \item \label{tor4g} $G_T = H_{4,g} := \left \langle \bigl(\begin{smallmatrix}
1 & 0 & 0\\ 
0 & -1 & 0\\ 
0 & 0 & -1
\end{smallmatrix}\bigr), \bigl(\begin{smallmatrix}
1 & 0 & 0\\ 
0 & 1 & 0\\ 
0 & 0 & -1
\end{smallmatrix}\bigr) \right \rangle$ : $T = \G_m \times T_{L_2/\Q} \times T_{L_3/\Q}$ and $C(T)=D_2D_3$.
        
    \item \label{tor4h} $G_T = H_{4,h} := \left \langle \bigl(\begin{smallmatrix}
1 & 0 & 0\\ 
0 & -1 & 0\\ 
0 & 0 & -1
\end{smallmatrix}\bigr), \bigl(\begin{smallmatrix}
-1 & 0 & 0\\ 
0 & 0 & -1\\ 
0 & -1 & 0
\end{smallmatrix}\bigr) \right \rangle$ : $T = T_{L_1/\Q} \times T_{L/L_1}$ and $C(T)=D_L$.

\item \label{tor4i} $G_T = H_{4,i} := \left \langle \bigl(\begin{smallmatrix}
1 & 0 & 0\\ 
0 & -1 & 0\\ 
0 & 0 & -1
\end{smallmatrix}\bigr), \bigl(\begin{smallmatrix}
1 & 0 & 0\\ 
0 & 0 & 1\\ 
0 & 1 & 0
\end{smallmatrix}\bigr) \right \rangle$ : $T = \G_m \times T_{L/L_1}$ and $\displaystyle C(T)=\frac{D_L}{D_1}=D_2D_3$.

\item \label{tor4j} $G_T = H_{4,j} := \left \langle \bigl(\begin{smallmatrix}
-1 & 0 & 0\\ 
0 & 1 & 0\\ 
0 & 0 & 1
\end{smallmatrix}\bigr), \bigl(\begin{smallmatrix}
1 & 0 & 0\\ 
0 & 0 & 1\\ 
0 & 1 & 0
\end{smallmatrix}\bigr) \right \rangle$ : $T = \res_{L_1/\Q}\G_m \times T_{L_2/\Q}$ and $C(T)=D_1D_2$.

\item \label{tor4k} $G_T = H_{4,k} := \left \langle \bigl(\begin{smallmatrix}
-1 & 0 & 0\\ 
0 & 0 & 1\\ 
0 & 1 & 0
\end{smallmatrix}\bigr), \bigl(\begin{smallmatrix}
-1 & 0 & 0\\ 
0 & -1 & 0\\ 
0 & 0 & -1
\end{smallmatrix}\bigr) \right \rangle$ : $T = T_{L_3/\Q} \times T_{L/L_2}$ and $C(T)=D_1D_3^2$.

\item \label{tor4l} $G_T = H_{4,l} := \left \langle \bigl(\begin{smallmatrix}
-1 & 0 & 0\\ 
0 & 0 & 1\\ 
0 & 1 & 0
\end{smallmatrix}\bigr), \bigl(\begin{smallmatrix}
-1 & 0 & 0\\ 
1 & 0 & -1\\ 
-1 & -1 & 0
\end{smallmatrix}\bigr) \right \rangle$ : 
$$
T=\left \{ (v_1, v_2) \in T_{L/L_1} \times T_{L/L_3} : v_1 \cdot hv_1 = v_2 \cdot hv_2 \right \}.
$$
Since
$$
1 \rightarrow T \rightarrow T_{L/L_1} \times T_{L/L_3} \xrightarrow{(v_1, v_2) \mapsto \frac{v_1 \cdot hv_1}{v_2 \cdot hv_2}} T_{L_2/\Q} \rightarrow 1
$$
is exact, we have $\displaystyle C(T)=\left ( \frac{D_L}{D_1} \cdot \frac{D_L}{D_3} \right ) \cdot D_2^{-1}=D_L$.

\item \label{tor4m} $G_T = H_{4,m} := \left \langle \bigl(\begin{smallmatrix}
-1 & 0 & 0\\ 
0 & 0 & 1\\ 
0 & 1 & 0
\end{smallmatrix}\bigr), \bigl(\begin{smallmatrix}
1 & 0 & 0\\ 
-1 & 0 & 1\\ 
1 & 1 & 0
\end{smallmatrix}\bigr) \right \rangle$ : $$
T=\left \{ (v_1, v_2) \in T_{L/L_1} \times \res_{L_3/\Q}\G_m : v_1 \cdot hv_2 = v_2 \cdot hv_1 \right \}.
$$
Since
$$
1 
\rightarrow T 
\rightarrow T_{L/L_1} \times \res_{L_3/\Q}\G_m 
\xrightarrow{(v_1, v_2) \mapsto \frac{v_1 \cdot hv_2}{v_2 \cdot hv_1}} T_{L_3/\Q} 
\rightarrow 1
$$
is exact, we have $\displaystyle C(T)=\left ( \frac{D_L}{D_1} \cdot D_3 \right ) \cdot D_3^{-1}=D_2D_3$.

\item \label{tor4n} $G_T = H_{4,n} := \left \langle \bigl(\begin{smallmatrix}
-1 & 0 & 0\\ 
0 & 0 & 1\\ 
0 & 1 & 0
\end{smallmatrix}\bigr), \bigl(\begin{smallmatrix}
-1 & 1 & -1\\ 
0 & 0 & -1\\ 
0 & -1 & 0
\end{smallmatrix}\bigr) \right \rangle$ : $T=T_{L/\Q}$ so $C(T)=D_L$. 

\item \label{tor4o} $G_T = H_{4,o} := \left \langle \bigl(\begin{smallmatrix}
-1 & 0 & 0\\ 
0 & 0 & 1\\ 
0 & 1 & 0
\end{smallmatrix}\bigr), \bigl(\begin{smallmatrix}
1 & -1 & 1\\ 
0 & 0 & 1\\ 
0 & 1 & 0
\end{smallmatrix}\bigr) \right \rangle$ : $T=\left \{ v \in \res_{L/\Q}\G_m : v \cdot gv = hv \cdot ghv \right \}$. Since
$$
1 \rightarrow T \rightarrow \res_{L/\Q}\G_m \xrightarrow{v \mapsto \frac{v \cdot gv}{hv \cdot ghv}} T_{L_1/\Q}
\rightarrow 1
$$
is exact, we have $C(T)=D_2D_3$. 

\end{enumerate}

\item $G_T=\left \langle g \right \rangle \cong C_6$ : $L_6 = L$ has the unique cubic subfield $L_3 = L^{g^3}$ and the unique quadratic subfield $L_2 = L^{g^2}$. $T$ is one of the following types. 

\begin{enumerate}[label=(\alph*)]
    \item \label{tor6a} $G_T = H_{6,a} := \left \langle \bigl(\begin{smallmatrix}
1 & 0 & 0\\ 
0 & 0 & -1\\ 
0 & 1 & 1
\end{smallmatrix}\bigr) \right \rangle$ : $T = \G_m \times (T_{L_6/L_2} \cap T_{L_6/L_3})$ and $\displaystyle C(T)=\frac{D_6}{D_2D_3}$.
        
    \item \label{tor6b} $G_T = H_{6,b} := \left \langle \bigl(\begin{smallmatrix}
-1 & 0 & 0\\ 
0 & 0 & 1\\ 
0 & -1 & -1
\end{smallmatrix}\bigr) \right \rangle$ : $T = T_{L_2/\Q} \times T_{L_3/\Q}$ and $\displaystyle C(T)=D_2D_3$.

    \item \label{tor6c} $G_T = H_{6,c} := \left \langle \bigl(\begin{smallmatrix}
-1 & 0 & 0\\ 
0 & 0 & 1\\ 
0 & -1 & 1
\end{smallmatrix}\bigr) \right \rangle$ : $T = T_{L_2/\Q} \times (T_{L_6/L_2} \cap T_{L_6/L_3})$ and $\displaystyle C(T)=\frac{D_6}{D_3}$.
        
    \item \label{tor6d} $G_T = H_{6,d} := \left \langle \bigl(\begin{smallmatrix}
0 & -1 & 0\\ 
0 & 0 & -1\\ 
-1 & 0 & 0
\end{smallmatrix}\bigr) \right \rangle$ : $T = T_{L_6/L_3}$ and $\displaystyle C(T)=\frac{D_6}{D_3}$.
    \end{enumerate}

\item $G_T=\left \langle g, h : g^3=h^2=(gh)^2=1 \right \rangle \cong S_3$ : $L_6 = L$ has $3$ isomorphic cubic subfields $L^h, L^{gh}$ and $L^{g^2h}$ (denoted by $L_3$) and the unique quadratic subfield $L_2 = L^g$. Then $D_6=D_3^2D_2$ by the equation (\ref{eq22b}). $T$ is one of the following types. 

\begin{enumerate}[label=(\alph*)]
\addtocounter{enumii}{4}

    \item \label{tor6e} $G_T = H_{6,e} := \left \langle \bigl(\begin{smallmatrix}
1 & 0 & 0\\ 
0 & 0 & -1\\ 
0 & 1 & -1
\end{smallmatrix}\bigr), \bigl(\begin{smallmatrix}
-1 & 0 & 0\\ 
0 & 0 & -1\\ 
0 & -1 & 0
\end{smallmatrix}\bigr) \right \rangle$ : $T = T_{L_2/\Q} \times (T_{L_6/L_2} \cap T_{L_6/L_3})$ and $C(T)=D_2D_3$.
        
    \item \label{tor6f} $G_T = H_{6,f} := \left \langle \bigl(\begin{smallmatrix}
1 & 0 & 0\\ 
0 & 0 & -1\\ 
0 & 1 & -1
\end{smallmatrix}\bigr), \bigl(\begin{smallmatrix}
1 & 0 & 0\\ 
0 & 0 & 1\\ 
0 & 1 & 0
\end{smallmatrix}\bigr) \right \rangle$ : $T = \G_m \times T_{L_3/\Q}$ and $C(T)=D_3$.

    \item \label{tor6g} $G_T = H_{6,g} := \left \langle \bigl(\begin{smallmatrix}
1 & 0 & 0\\ 
0 & 0 & -1\\ 
0 & 1 & -1
\end{smallmatrix}\bigr), \bigl(\begin{smallmatrix}
-1 & 0 & 0\\ 
0 & 0 & 1\\ 
0 & 1 & 0
\end{smallmatrix}\bigr) \right \rangle$ : $T =T_{L_2/\Q} \times T_{L_3/\Q}$ and $C(T)=D_2D_3$.
        
    \item \label{tor6h} $G_T = H_{6,h} := \left \langle \bigl(\begin{smallmatrix}
1 & 0 & 0\\ 
0 & 0 & -1\\ 
0 & 1 & -1
\end{smallmatrix}\bigr), \bigl(\begin{smallmatrix}
1 & 0 & 0\\ 
0 & 0 & -1\\ 
0 & -1 & 0
\end{smallmatrix}\bigr) \right \rangle$ : $T = \G_m \times (T_{L_6/L_2} \cap T_{L_6/L_3})$ and $C(T)=D_3$.

\item \label{tor6i} $G_T = H_{6,i} := \left \langle \bigl(\begin{smallmatrix}
0 & 1 & 0\\ 
0 & 0 & 1\\ 
1 & 0 & 0
\end{smallmatrix}\bigr), \bigl(\begin{smallmatrix}
0 & 0 & -1\\ 
0 & -1 & 0\\ 
-1 & 0 & 0
\end{smallmatrix}\bigr) \right \rangle$ : $T = T_{L_6/L_3}$ and $C(T)=D_2D_3$.

\item \label{tor6j} $G_T = H_{6,j} := \left \langle \bigl(\begin{smallmatrix}
0 & 1 & 0\\ 
0 & 0 & 1\\ 
1 & 0 & 0
\end{smallmatrix}\bigr), \bigl(\begin{smallmatrix}
0 & 0 & 1\\ 
0 & 1 & 0\\ 
1 & 0 & 0
\end{smallmatrix}\bigr) \right \rangle$ : $T = \res_{L_3/\Q}\G_m$ and $C(T)=D_3$.

\end{enumerate}

\item $G_T=\left \langle g,h : g^4=h^2=1, \, gh=hg \right \rangle \cong C_4 \times C_2$ :  
$$
L_4=L^h, \, L_4'=L^{g^2h}, \, L_4''= L^{g^2}
$$
are quartic subfields of $L$ and
$$
L_2=L^g, \, L_2'=L^{\left \langle g^2, h \right \rangle}, \, L_2''= L^{gh}
$$
are quadratic subfields of $L$. Since $\Gal(L/L_2') \cong C_2^2$ and $\Gal(L_4''/\Q) \cong C_2^2$, we have the following formulas by the equation (\ref{eq22a}):
\begin{subequations} 
\begin{align}
   D_LD_2'^2 &= D_4D_4'D_4'' \label{eq35a} \\
   D_4'' &= D_2D_2'D_2''. \label{eq35b}
  \end{align}
\end{subequations}
$T$ is one of the following types. 

\begin{enumerate}[label=(\alph*)]
    \item \label{tor8a} $G_T = H_{8,a} := \left \langle \bigl(\begin{smallmatrix}
1 & 0 & 0\\ 
0 & 0 & -1\\ 
0 & 1 & 0
\end{smallmatrix}\bigr), \bigl(\begin{smallmatrix}
-1 & 0 & 0\\ 
0 & -1 & 0\\ 
0 & 0 & -1
\end{smallmatrix}\bigr) \right \rangle$ : $T = T_{L_2/\Q} \times T_{L_4'/L_2'}$ and $\displaystyle C(T)=D_2 \frac{D_4'}{D_2'}$.
        
    \item \label{tor8b} $G_T = H_{8,b} := \left \langle \bigl(\begin{smallmatrix}
1 & 0 & 1\\ 
0 & 0 & -1\\ 
0 & 1 & 0
\end{smallmatrix}\bigr), \bigl(\begin{smallmatrix}
-1 & 0 & 0\\ 
0 & -1 & 0\\ 
0 & 0 & -1
\end{smallmatrix}\bigr) \right \rangle$ : $T = T_{L/L_4} \cap T_{L/L_2''}$ and $\displaystyle C(T)=\frac{D_L}{D_4D_2''}=D_2 \frac{D_4'}{D_2'}$ by the equations (\ref{eq35a}) and (\ref{eq35b}).
    \end{enumerate}

\item $G_T=\left \langle g_1, g_2, g_3 \right \rangle \cong C_2^3$ : For $\left \{ i,j,k \right \} = \left \{ 1,2,3 \right \}$, 
$$
M_i = L^{g_i}, \, M_{ij} = L^{g_ig_j}, \,  M=L^{g_1g_2g_3}
$$
are quartic subfields of $L$ and
$$
K_i = L^{\left \langle g_i, g_jg_k \right \rangle}, \, 
K_{ij}=L^{\left \langle g_i, g_j \right \rangle}, \,  
K=L^{\left \langle g_1g_2, g_1g_3 \right \rangle}
$$
are quadratic subfields of $L$. Since $M_{ij}/\Q$, $M/\Q$ and $L/K_3$ are Galois extensions with Galois groups $C_2^2$, we have the following formulas by the equation (\ref{eq22a}):
\begin{subequations} 
\begin{align}
 D_{M_{ij}} &= D_{K_k}D_{K_{ij}}D_K \label{eq35c} \\
 D_M &= D_{K_1}D_{K_2}D_{K_3} \label{eq35d} \\
 D_LD_{K_3}^2 &= D_{M_3}D_{M_{12}}D_M. \label{eq35e}
\end{align}
\end{subequations}
$T$ is one of the following types.

\begin{enumerate}[label=(\alph*)]
\addtocounter{enumii}{2}

    \item \label{tor8c} $G_T = H_{8,c} := \left \langle \bigl(\begin{smallmatrix}
1 & 0 & 0\\ 
0 & -1 & 0\\
0 & 0 & -1
\end{smallmatrix}\bigr), \bigl(\begin{smallmatrix}
-1 & 0 & 0\\ 
0 & -1 & 0\\ 
0 & 0 & 1
\end{smallmatrix}\bigr), \bigl(\begin{smallmatrix}
-1 & 0 & 0\\ 
0 & -1 & 0\\ 
0 & 0 & -1
\end{smallmatrix}\bigr) \right \rangle$ : $T = T_{K_1/\Q} \times T_{K/\Q} \times T_{K_2/\Q}$ and $\displaystyle C(T)=D_KD_{K_1}D_{K_2}$.
        
    \item \label{tor8d} $G_T = H_{8,d} := \left \langle \bigl(\begin{smallmatrix}
1 & 0 & 0\\ 
0 & -1 & 0\\ 
0 & 0 & -1
\end{smallmatrix}\bigr), \bigl(\begin{smallmatrix}
-1 & 0 & 0\\ 
0 & 0 & -1\\ 
0 & -1 & 0
\end{smallmatrix}\bigr), \bigl(\begin{smallmatrix}
-1 & 0 & 0\\ 
0 & -1 & 0\\ 
0 & 0 & -1
\end{smallmatrix}\bigr) \right \rangle$ : $T = T_{K_1/\Q} \times T_{M_{13}/K_{13}}$ and $\displaystyle C(T)=D_{K_1} \frac{D_{M_{13}}}{D_{K_{13}}} = D_KD_{K_1}D_{K_2}$.

    \item \label{tor8e} $G_T = H_{8,e} := \left \langle \bigl(\begin{smallmatrix}
-1 & 0 & 0\\ 
0 & 0 & 1\\ 
0 & 1 & 0
\end{smallmatrix}\bigr), \bigl(\begin{smallmatrix}
-1 & 0 & 0\\ 
1 & 0 & -1\\ 
-1 & -1 & 0
\end{smallmatrix}\bigr), \bigl(\begin{smallmatrix}
-1 & 0 & 0\\ 
0 & -1 & 0\\ 
0 & 0 & -1
\end{smallmatrix}\bigr) \right \rangle$ :
$$
T=\left \{ (v_1, v_2) \in T_{M_{13}/K_{13}} \times T_{M/K_3} : v_1 \cdot g_2v_1 = v_2 \cdot g_2v_2 \right \}.
$$
Since 
$$
1 \rightarrow T \rightarrow T_{M_{13}/K_{13}} \times T_{M/K_3} \xrightarrow{(v_1, v_2) \mapsto \frac{v_1 \cdot g_2v_1}{v_2 \cdot g_2v_2}} T_{K_2/\Q} \rightarrow 1
$$
is exact, we have $\displaystyle C(T)=\frac{D_{M_{13}}}{D_{K_{13}}} \cdot \frac{D_M}{D_{K_3}} \cdot D_{K_2}^{-1} = D_K D_{K_1}D_{K_2}$ by the equations (\ref{eq35c}) and (\ref{eq35d}).

    \item \label{tor8f} $G_T = H_{8,f} := \left \langle \bigl(\begin{smallmatrix}
-1 & 0 & 0\\ 
0 & 0 & 1\\ 
0 & 1 & 0
\end{smallmatrix}\bigr), \bigl(\begin{smallmatrix}
-1 & 1 & -1\\ 
0 & 0 & -1\\ 
0 & -1 & 0
\end{smallmatrix}\bigr), \bigl(\begin{smallmatrix}
-1 & 0 & 0\\ 
0 & -1 & 0\\ 
0 & 0 & -1
\end{smallmatrix}\bigr) \right \rangle$ : $T = T_{L/K_{12}} \cap T_{L/M_3}$ and $\displaystyle C(T)
=\frac{D_L}{D_{K_{12}}D_{M_3}}
=D_KD_{K_1}D_{K_2}$ by the equations (\ref{eq35c}), (\ref{eq35d}) and (\ref{eq35e}).

\end{enumerate}

\item \label{item3ix} $G_T=\left \langle g, h : g^4=h^2=(gh)^2=1 \right \rangle \cong D_4$ : 
$$
M=L^{g^2}, \, M_1=L^h, \, M_1'=L^{g^2h}, \, M_2=L^{gh}, \, M_2'=L^{g^3h}
$$
are quartic subfields of $L$ ($M_i'$ are Galois conjugates of $M_i$ so $D_{M_i}=D_{M_i'}$) and 
$$
K=L^g, \, K_1=L^{\left \langle g^2, h \right \rangle}, \, K_2=L^{\left \langle g^2, gh \right \rangle}
$$
are quadratic subfields of $L$. Then we have the following lattice of subfields of $L$. 
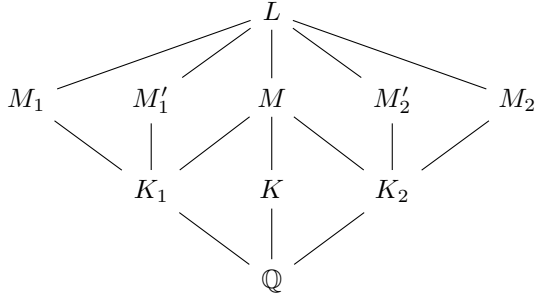
\[ 
\begin{tikzcd}
  &  & L \arrow[d, no head] \arrow[rd, no head] \arrow[rrd, no head] \arrow[ld, no head] &  &  \\
M_1 \arrow[rru, no head] \arrow[rd, no head] & M_1' \arrow[d, no head] & M \arrow[ld, no head] \arrow[d, no head] \arrow[rd, no head]  & M_2' \arrow[d, no head] & M_2 \arrow[ld, no head] \\
  & K_1 \arrow[rd, no head] & K \arrow[d, no head]   & K_2 \arrow[ld, no head] &  \\
  &  & \Q  &   &                        
\end{tikzcd}
\]
Since $L/K_1$, $L/K_2$ and $M/\Q$ are Galois extensions with Galois groups $C_2^2$, we have the following formulas by the equation (\ref{eq22a}):
\begin{subequations} 
\begin{align}
 D_LD_{K_1}^2 &= D_{M_1}^2D_M \label{eq35f} \\
 D_LD_{K_2}^2 &= D_{M_2}^2D_M \label{eq35g} \\
 D_M &= D_KD_{K_1}D_{K_2} \label{eq35h} \\
 \frac{D_{M_1}}{D_{K_1}} &= \frac{D_{M_2}}{D_{K_2}}. \label{eq35i}
\end{align}
\end{subequations}
The last equation is a consequence of the equations (\ref{eq35f}) and (\ref{eq35g}) (cf. \cite[Proposition 2.4]{ASVW21}.) $T$ is one of the following types. 

\begin{enumerate}[label=(\alph*)]
\addtocounter{enumii}{6}

    \item \label{tor8g} $G_T = H_{8,g} := \left \langle \bigl(\begin{smallmatrix}
1 & 0 & 0\\ 
0 & 0 & -1\\ 
0 & 1 & 0
\end{smallmatrix}\bigr), \bigl(\begin{smallmatrix}
-1 & 0 & 0\\ 
0 & 0 & 1\\ 
0 & 1 & 0
\end{smallmatrix}\bigr) \right \rangle$ : $T = T_{K/\Q} \times T_{M_2'/K_2}$ and $\displaystyle C(T)=\frac{D_KD_{M_2}}{D_{K_2}}$.
        
    \item \label{tor8h} $G_T = H_{8,h} := \left \langle \bigl(\begin{smallmatrix}
1 & 0 & 0\\ 
0 & 0 & -1\\ 
0 & 1 & 0
\end{smallmatrix}\bigr), \bigl(\begin{smallmatrix}
1 & 0 & 0\\ 
0 & 0 & -1\\ 
0 & -1 & 0
\end{smallmatrix}\bigr) \right \rangle$ : $T = \G_m \times T_{M_2/K_2}$ and $\displaystyle C(T)=\frac{D_{M_2}}{D_{K_2}}$.

    \item \label{tor8i} $G_T = H_{8,i} := \left \langle \bigl(\begin{smallmatrix}
-1 & 0 & 0\\ 
0 & 0 & 1\\ 
0 & -1 & 0
\end{smallmatrix}\bigr), \bigl(\begin{smallmatrix}
-1 & 0 & 0\\ 
0 & 0 & 1\\ 
0 & 1 & 0
\end{smallmatrix}\bigr) \right \rangle$ : $T = T_{K_2/\Q} \times T_{M_2/K_2}$ and $\displaystyle C(T)=D_{M_2}$.

    \item \label{tor8j} $G_T = H_{8,j} := \left \langle \bigl(\begin{smallmatrix}
-1 & 0 & 0\\ 
0 & 0 & 1\\ 
0 & -1 & 0
\end{smallmatrix}\bigr), \bigl(\begin{smallmatrix}
1 & 0 & 0\\ 
0 & 0 & -1\\ 
0 & -1 & 0
\end{smallmatrix}\bigr) \right \rangle$ : $T = T_{K_1/\Q} \times T_{M_2'/K_2}$ and $\displaystyle C(T)=\frac{D_{K_1}D_{M_2}}{D_{K_2}}=D_{M_1}$.

    \item \label{tor8k} $G_T = H_{8,k} := \left \langle \bigl(\begin{smallmatrix}
1 & 0 & 1\\ 
0 & 0 & -1\\ 
0 & 1 & 0
\end{smallmatrix}\bigr), \bigl(\begin{smallmatrix}
-1 & 0 & 0\\ 
0 & 0 & -1\\ 
0 & -1 & 0
\end{smallmatrix}\bigr) \right \rangle$ : $T = T_{L/K_1} \cap T_{L/M_2'}$ and $\displaystyle C(T)=\frac{D_L}{D_{K_1}D_{M_2}}=\frac{D_KD_{M_2}}{D_{K_2}}$ by the equations (\ref{eq35g}) and (\ref{eq35h}).

    \item \label{tor8l} $G_T = H_{8,l} := \left \langle \bigl(\begin{smallmatrix}
1 & 0 & 1\\ 
0 & 0 & -1\\ 
0 & 1 & 0
\end{smallmatrix}\bigr), \bigl(\begin{smallmatrix}
1 & 0 & 0\\ 
0 & 0 & 1\\ 
0 & 1 & 0
\end{smallmatrix}\bigr) \right \rangle$ : $
T = \left \{ v \in \res_{M_2'/\Q}\G_m : v \cdot g^2v = gv \cdot g^3v \right \}$. Since
$$
1 \rightarrow T \rightarrow \res_{M_2'/\Q}\G_m
\xrightarrow{v \mapsto \frac{v \cdot g^2v}{gv \cdot g^3v}} T_{K_2/\Q} \rightarrow 1
$$
is exact, we have $\displaystyle C(T)=\frac{D_{M_2}}{D_{K_2}}$.

    \item \label{tor8m} $G_T = H_{8,m} := \left \langle \bigl(\begin{smallmatrix}
-1 & 0 & -1\\ 
0 & 0 & 1\\ 
0 & -1 & 0
\end{smallmatrix}\bigr), \bigl(\begin{smallmatrix}
-1 & 0 & 0\\ 
0 & 0 & -1\\ 
0 & -1 & 0
\end{smallmatrix}\bigr) \right \rangle$ : $T = T_{M_2'/\Q}$ and $\displaystyle C(T)=D_{M_2}$.

    \item \label{tor8n} $G_T = H_{8,n} := \left \langle \bigl(\begin{smallmatrix}
-1 & 0 & -1\\ 
0 & 0 & 1\\ 
0 & -1 & 0
\end{smallmatrix}\bigr), \bigl(\begin{smallmatrix}
1 & 0 & 0\\ 
0 & 0 & 1\\ 
0 & 1 & 0
\end{smallmatrix}\bigr) \right \rangle$ : $T = T_{L/M_2'} \cap T_{L/K}$ and $\displaystyle C(T)=\frac{D_L}{D_{M_2}D_K}=D_{M_1}$ by the equations (\ref{eq35g}), (\ref{eq35h}) and (\ref{eq35i}).

\end{enumerate}

\item $G_T=\left \langle g, h : g^6=h^2=1, \, gh=hg \right \rangle \cong C_6 \times C_2$ : $L_6 = L^{g^3h}$ has the unique cubic subfield $L_3=L^{\left \langle g^3, h \right \rangle}$ and the unique quadratic subfield $L_2=L^{\left \langle g^2, gh \right \rangle}$. Also $L_2'=L^g$ is a quadratic subfield of $L$ such that $L=L_6L_2'$. $T$ is the following type.

\begin{enumerate}[label=(\alph*)]

\item \label{tor12a} $G_T=H_{12, a}:= \left \langle \bigl(\begin{smallmatrix}
1 & 0 & 0\\ 
0 & 0 & -1\\ 
0 & 1 & 1
\end{smallmatrix}\bigr), \bigl(\begin{smallmatrix}
-1 & 0 & 0\\ 
0 & -1 & 0\\ 
0 & 0 & -1
\end{smallmatrix}\bigr) \right \rangle$ : $T=T_{L_2'/\Q} \times (T_{L_6/L_3} \cap T_{L_6/L_2})$ and $\displaystyle C(T)= D_2' \frac{D_6}{D_3D_2}$.

\end{enumerate}

\item \label{item3xi} $G_T=\left \langle g, h : g^6=h^2=(gh)^2=1 \right \rangle \cong D_6$ : Consider the following lattice of subfields of $L$.
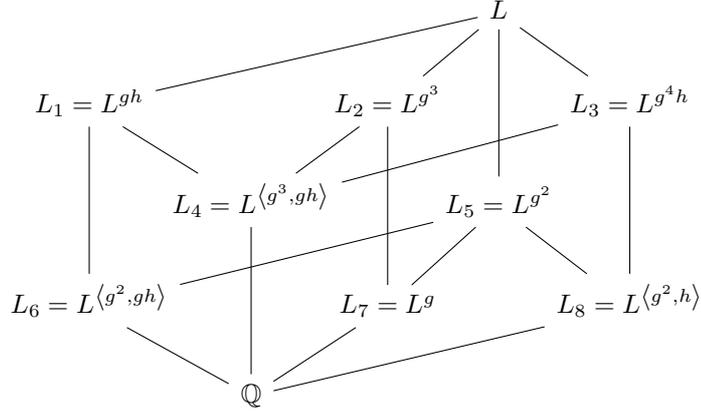
\[ 
\begin{tikzcd}[column sep=-0.7em]
& & & L \arrow[rd, no head] \arrow[dd, no head] & \\
L_1 = L^{gh} \arrow[rrru, no head] \arrow[dd, no head] \arrow[rd, no head] 
& & L_2 = L^{g^3} \arrow[ru, no head] \arrow[dd, no head] \arrow[ld, no head] 
& & L_3 = L^{g^4h} \arrow[dd, no head] \\
& L_4 = L^{\left \langle g^3, gh \right \rangle} \arrow[rrru, no head] \arrow[dd, no head] 
& & L_5 = L^{g^2} \arrow[rd, no head] & \\
L_6 = L^{\left \langle g^2, gh \right \rangle} \arrow[rd, no head] \arrow[rrru, no head] 
& & L_7 = L^{g} \arrow[ld, no head] \arrow[ru, no head] 
& & L_8 = L^{\left \langle g^2, h \right \rangle} \arrow[llld, no head] \\
& \Q & & & 
\end{tikzcd}
\]

Since $L_2/\Q$, $L/L_6$ and $L/L_8$ are Galois extensions with Galois groups $S_3$, we have the following formulas by the equation (\ref{eq22b}): 
\begin{subequations} 
\begin{align}
   D_2 &= D_4^2D_7 \label{eq36a} \\
   \frac{D_1^2}{D_6^2} & = \frac{D_L}{D_5} = \frac{D_3^2}{D_8^2}. \label{eq36b}
  \end{align}
\end{subequations}
$T$ is one of the following types. 

\begin{enumerate}[label=(\alph*)]
\addtocounter{enumii}{1}

\item \label{tor12b} $G_T=H_{12, b}:= \left \langle \bigl(\begin{smallmatrix}
1 & 0 & 0\\ 
0 & 0 & -1\\ 
0 & 1 & 1
\end{smallmatrix}\bigr), \bigl(\begin{smallmatrix}
-1 & 0 & 0\\ 
0 & 0 & 1\\ 
0 & 1 & 0
\end{smallmatrix}\bigr) \right \rangle$ : $T=T_{L_7/\Q} \times (T_{L_1/L_4} \cap T_{L_1/L_6})$ and $\displaystyle C(T)=\frac{D_1D_7}{D_4D_6}$.

\item \label{tor12c} $G_T=H_{12, c}:= \left \langle \bigl(\begin{smallmatrix}
1 & 0 & 0\\ 
0 & 0 & -1\\ 
0 & 1 & 1
\end{smallmatrix}\bigr), \bigl(\begin{smallmatrix}
1 & 0 & 0\\ 
0 & 0 & -1\\ 
0 & -1 & 0
\end{smallmatrix}\bigr) \right \rangle$ : $T=\G_m \times (T_{L_3/L_4} \cap T_{L_3/L_8})$ and $\displaystyle C(T)=\frac{D_3}{D_4D_8}=\frac{D_1}{D_4D_6}$.

\item \label{tor12d} $G_T=H_{12, d}:= \left \langle \bigl(\begin{smallmatrix}
-1 & 0 & 0\\ 
0 & 0 & 1\\ 
0 & -1 & -1
\end{smallmatrix}\bigr), \bigl(\begin{smallmatrix}
-1 & 0 & 0\\ 
0 & 0 & 1\\ 
0 & 1 & 0
\end{smallmatrix}\bigr) \right \rangle$ : $T=T_{L_6/\Q} \times (T_{L_2/L_4} \cap T_{L_2/L_7})$ and $\displaystyle C(T)=\frac{D_2D_6}{D_4D_7}=D_4D_6$.

\item \label{tor12e} $G_T=H_{12, e}:= \left \langle \bigl(\begin{smallmatrix}
-1 & 0 & 0\\ 
0 & 0 & 1\\ 
0 & -1 & -1
\end{smallmatrix}\bigr), \bigl(\begin{smallmatrix}
1 & 0 & 0\\ 
0 & 0 & -1\\ 
0 & -1 & 0
\end{smallmatrix}\bigr) \right \rangle$ : $T=T_{L_8/\Q} \times T_{L_4/\Q}$ and $\displaystyle C(T)=D_4D_8$.

\item \label{tor12f} $G_T=H_{12, f}:= \left \langle \bigl(\begin{smallmatrix}
-1 & 0 & 0\\ 
0 & 0 & 1\\ 
0 & -1 & 1
\end{smallmatrix}\bigr), \bigl(\begin{smallmatrix}
-1 & 0 & 0\\ 
0 & 0 & -1\\ 
0 & -1 & 0
\end{smallmatrix}\bigr) \right \rangle$ : $T=T_{L_6/\Q} \times (T_{L_1/L_4} \cap T_{L_1/L_6})$ and $\displaystyle C(T)=\frac{D_1}{D_4}$.

\item \label{tor12g} $G_T=H_{12, g}:= \left \langle \bigl(\begin{smallmatrix}
-1 & 0 & 0\\ 
0 & 0 & 1\\ 
0 & -1 & 1
\end{smallmatrix}\bigr), \bigl(\begin{smallmatrix}
-1 & 0 & 0\\ 
0 & 0 & 1\\ 
0 & 1 & 0
\end{smallmatrix}\bigr) \right \rangle$ : $T=T_{L_6/\Q} \times (T_{L_3/L_4} \cap T_{L_3/L_8})$ and $\displaystyle C(T)=\frac{D_3D_6}{D_4D_8}=\frac{D_1}{D_4}$.

\item \label{tor12h} $G_T=H_{12, h}:= \left \langle \bigl(\begin{smallmatrix}
0 & -1 & 0\\ 
0 & 0 & -1\\ 
-1 & 0 & 0
\end{smallmatrix}\bigr), \bigl(\begin{smallmatrix}
0 & 0 & -1\\ 
0 & -1 & 0\\ 
-1 & 0 & 0
\end{smallmatrix}\bigr) \right \rangle$ : $T=T_{L_1/L_4}$ and $\displaystyle C(T)=\frac{D_1}{D_4}$.
\end{enumerate}

\item $G_T=\left \langle g, h : g^3=h^2=(gh)^3=1 \right \rangle \cong A_4$ : Let
$$
L_6=L^h, \, L_6'=L^{g^{-1}hg}, \, L_4=L^{ghg}, \, L_4'=L^{gh} \text{ and } L_3=L^{\left \langle h, g^{-1}hg \right \rangle}.
$$
We have $D_6=D_6'$, $D_4=D_4'$ and the following formulas by the equations (\ref{eq22c2}) and (\ref{eq22c3}):
\begin{subequations} 
\begin{align}
   D_6 &= D_3D_4 \label{eq36e} \\
   D_L &= D_4^2D_6. \label{eq36f}
  \end{align}
\end{subequations}
$T$ is one of the following types. 

\begin{enumerate}[label=(\alph*)]
\addtocounter{enumii}{8}

\item \label{tor12i} $G_T=H_{12, i}:= \left \langle \bigl(\begin{smallmatrix}
0 & 1 & 0\\ 
0 & 0 & 1\\ 
1 & 0 & 0
\end{smallmatrix}\bigr), \bigl(\begin{smallmatrix}
-1 & 0 & 0\\ 
0 & 1 & 0\\ 
0 & 0 & -1
\end{smallmatrix}\bigr) \right \rangle$ : $T=T_{L_6/L_3}$ and $\displaystyle C(T)=\frac{D_6}{D_3}=D_4$.

\item \label{tor12j} $G_T=H_{12, j}:= \left \langle \bigl(\begin{smallmatrix}
0 & 1 & 0\\ 
0 & 0 & 1\\ 
1 & 0 & 0
\end{smallmatrix}\bigr), \bigl(\begin{smallmatrix}
0 & -1 & 1\\ 
0 & -1 & 0\\ 
1 & -1 & 0
\end{smallmatrix}\bigr) \right \rangle$ : $T=T_{L_4/\Q}$ and $\displaystyle C(T)=D_4$. 

\item \label{tor12k} $G_T=H_{12, k}:= \left \langle \bigl(\begin{smallmatrix}
0 & 1 & 0\\ 
0 & 0 & 1\\ 
1 & 0 & 0
\end{smallmatrix}\bigr), \bigl(\begin{smallmatrix}
-1 & -1 & -1\\ 
0 & 0 & 1\\ 
0 & 1 & 0
\end{smallmatrix}\bigr) \right \rangle$ : $T = \left \{ v \in T_{L/L_6'} : hv \cdot gv = g^2v \right \}$. Since
$$
1 \rightarrow T \rightarrow T_{L/L_6'}
\xrightarrow{v \mapsto \frac{hv \cdot gv}{g^2v}} T_{L_4'/\Q} \rightarrow 1
$$
is exact, we have $\displaystyle C(T)=\frac{D_L}{D_6D_4}=D_4$.

\end{enumerate}

\item $G_T=\left \langle g, h : g^4=h^2=(gh)^2=1 \right \rangle \times \left \langle i \right \rangle \cong D_4 \times C_2$ : $L_8=L^{hgi}$ is an octic subfield of $L$, 
$$
L_4=L^{\left \langle hg, i \right \rangle}, \, 
L_4'=L^{\left \langle hg, g^2i \right \rangle}, \, 
L_4''=L^{\left \langle gh, g^2i \right \rangle}, \,
L_4'''=L^{\left \langle hgi, g^2 \right \rangle}
$$
are quartic subfields of $L$ and
$$
L_2=L^{\left \langle h, gi \right \rangle}, \,
L_2'=L^{\left \langle hg, g^2, i \right \rangle}, \,
L_2''=L^{\left \langle g, hi \right \rangle}
$$
are quadratic subfields of $L$. $L_4''$ is a Galois conjugate of $L_4'$ so $D_4''=D_4'$. Since $L_8/L_2'$ and $L_4'''/\Q$ are Galois extensions with Galois groups $C_2^2$, we have the following formulas by the equation (\ref{eq22a}):
\begin{subequations} 
\begin{align}
   D_8D_2'^2 &= D_4D_4'D_4''' \label{eq37a} \\
   D_4''' &= D_2D_2'D_2''. \label{eq37b}
  \end{align}
\end{subequations}
$T$ is one of the following types. 

\begin{enumerate}[label=(\alph*)]

    \item \label{tor16a} $G_T = H_{16,a} := \left \langle \bigl(\begin{smallmatrix}
1 & 0 & 0\\ 
0 & 0 & -1\\
0 & 1 & 0
\end{smallmatrix}\bigr), \bigl(\begin{smallmatrix}
-1 & 0 & 0\\ 
0 & 0 & 1\\
0 & 1 & 0
\end{smallmatrix}\bigr), \bigl(\begin{smallmatrix}
-1 & 0 & 0\\ 
0 & -1 & 0\\ 
0 & 0 & -1
\end{smallmatrix}\bigr) \right \rangle$ : $T=T_{L_2''/\Q} \times T_{L_4'/L_2'}$ and $\displaystyle C(T)=D_2''\frac{D_4'}{D_2'}$. 
        
    \item \label{tor16b} $G_T = H_{16, b} := \left \langle \bigl(\begin{smallmatrix}
1 & 0 & 1\\ 
0 & 0 & -1\\ 
0 & 1 & 0
\end{smallmatrix}\bigr), \bigl(\begin{smallmatrix}
-1 & 0 & 0\\ 
0 & 0 & -1\\ 
0 & -1 & 0
\end{smallmatrix}\bigr), \bigl(\begin{smallmatrix}
-1 & 0 & 0\\ 
0 & -1 & 0\\ 
0 & 0 & -1
\end{smallmatrix}\bigr) \right \rangle$ : $T=T_{L_8/L_4} \cap T_{L_8/L_2}$ and $\displaystyle C(T)=\frac{D_8}{D_4D_2} = D_2''\frac{D_4'}{D_2'}$ by the equations (\ref{eq37a}) and (\ref{eq37b}).

\end{enumerate}

\item $G_T=\left \langle g, h : g^3=h^2=(gh)^3=1 \right \rangle \times \left \langle i \right \rangle \cong A_4 \times C_2$ : Let $L_{12} = L^{hi}$, $L_8=L^{hg}$, $L_6=L^{\left \langle hi, ghg^{-1} \right \rangle}$, $L_6'=L^{\left \langle h, i \right \rangle}$, $L_4=L^{\left \langle hg, i \right \rangle}$, $L_3=L^{\left \langle h, ghg^{-1}, i \right \rangle}$ and $L_2=L^{\left \langle g,h \right \rangle}$. $T$ is one of the following types. 

\begin{enumerate}[label=(\alph*)]

    \item \label{tor24a} $G_T = H_{24,a} := \left \langle \bigl(\begin{smallmatrix}
0 & 1 & 0\\ 
0 & 0 & 1\\
1 & 0 & 0
\end{smallmatrix}\bigr), \bigl(\begin{smallmatrix}
-1 & 0 & 0\\ 
0 & 1 & 0\\
0 & 0 & -1
\end{smallmatrix}\bigr), \bigl(\begin{smallmatrix}
-1 & 0 & 0\\ 
0 & -1 & 0\\ 
0 & 0 & -1
\end{smallmatrix}\bigr) \right \rangle$ : $T = T_{L_6/L_3}$ and $\displaystyle C(T) = \frac{D_6}{D_3}$.
        
    \item \label{tor24b} $G_T = H_{24,b} := \left \langle \bigl(\begin{smallmatrix}
0 & 1 & 0\\ 
0 & 0 & 1\\
1 & 0 & 0
\end{smallmatrix}\bigr), \bigl(\begin{smallmatrix}
0 & -1 & 1\\ 
0 & -1 & 0\\
1 & -1 & 0
\end{smallmatrix}\bigr), \bigl(\begin{smallmatrix}
-1 & 0 & 0\\ 
0 & -1 & 0\\ 
0 & 0 & -1
\end{smallmatrix}\bigr) \right \rangle$ : $T = T_{L_8/L_4} \cap T_{L_8/L_2}$ and $\displaystyle C(T) = \frac{D_8}{D_4D_2}$.

    \item \label{tor24c} $G_T = H_{24,c} := \left \langle \bigl(\begin{smallmatrix}
0 & 1 & 0\\ 
0 & 0 & 1\\
1 & 0 & 0
\end{smallmatrix}\bigr), \bigl(\begin{smallmatrix}
-1 & -1 & -1\\ 
0 & 0 & 1\\
0 & 1 & 0
\end{smallmatrix}\bigr), \bigl(\begin{smallmatrix}
-1 & 0 & 0\\ 
0 & -1 & 0\\ 
0 & 0 & -1
\end{smallmatrix}\bigr) \right \rangle$ : $T = \left \{ v \in T_{L_{12}/L_6'} : v \cdot hgv = g^2v \right \}$. Since
$$
1 \rightarrow T \rightarrow T_{L_{12}/L_6'}
\xrightarrow{v \mapsto \frac{v \cdot hgv}{g^2v}} T_{L_8/L_4} \cap T_{L_8/L_2} \rightarrow 1
$$
is exact, we have $\displaystyle C(T) = \frac{D_{12}D_4D_2}{D_8D_6'}$.

\end{enumerate}

One can prove that the number $C(T)$ for each torus of \ref{tor24a}, \ref{tor24b} and \ref{tor24c} are equal. 

\begin{lemma} \label{lem3b}
$\displaystyle \frac{D_6}{D_3} = \frac{D_8}{D_4D_2} = \frac{D_{12}D_4D_2}{D_8D_6'}$.
\end{lemma}

\begin{proof}
We need to consider some more subfields of $L$ for the proof. Let
$$
L_{12}'=L^h, \, L_{12}''=L^i, \, L_6''=L^{\left \langle h, ghg^{-1} \right \rangle} \text{ and } L_6'''=L^{\left \langle hi, g^{-1}hg \right \rangle}.
$$
Since $\left \langle \tilde{h}i, g^{-1} \tilde{h} g \right \rangle = \left \langle hi, ghg^{-1} \right \rangle$ for $\tilde{h} := g^{-1}hg$, $L_6'''$ is a Galois conjugate of $L_6$ so $D_6'''=D_6$. Consider the following lattices of subfields. 
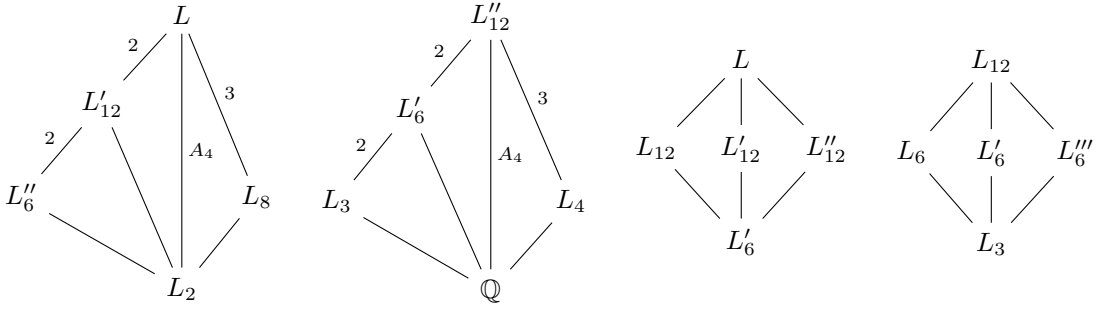
\[ 
\begin{tikzcd}[column sep=0.8em]
&  & L \arrow[ld, "2"', no head] \arrow[rdd, "3", no head] \arrow[ddd, "A_4", no head] & \\
& L_{12}' \arrow[ld, "2"', no head] \arrow[rdd, no head] &  & \\
L_6'' \arrow[rrd, no head] &   &   & L_8 \arrow[ld, no head] \\
&  & L_2  &   
\end{tikzcd}
\quad
\begin{tikzcd}[column sep=0.8em]
&  & L_{12}'' \arrow[ld, "2"', no head] \arrow[rdd, "3", no head] \arrow[ddd, "A_4", no head] & \\
& L_6' \arrow[ld, "2"', no head] \arrow[rdd, no head] &  & \\
L_3 \arrow[rrd, no head] &   &   & L_4 \arrow[ld, no head] \\
&  & \Q  &   
\end{tikzcd}
\quad
\begin{tikzcd}[column sep=0.8em]
& L \arrow[ld, no head] \arrow[d, no head] \arrow[rd, no head] & \\
L_{12} \arrow[rd, no head] & L_{12}' \arrow[d, no head]  & L_{12}'' \arrow[ld, no head] \\
& L_6'  & 
\end{tikzcd}
\quad 
\begin{tikzcd}[column sep=0.8em]
& L_{12} \arrow[ld, no head] \arrow[d, no head] \arrow[rd, no head] & \\
L_6 \arrow[rd, no head] & L_6' \arrow[d, no head]  & L_6''' \arrow[ld, no head] \\
& L_3  & 
\end{tikzcd}
\]
By the above diagrams, we have the following formulas by the equations (\ref{eq22a}) and (\ref{eq22c3}):
\begin{subequations} 
\begin{align}
   D_LD_2^2 &= D_{12}'D_8^2 \label{eq38a} \\
   D_{12}'' &= D_6'D_4^2 \label{eq38b} \\
   D_LD_6'^2 &= D_{12}D_{12}'D_{12}'' \label{eq38c} \\
   D_{12}D_3^2 &= D_6^2 D_6'. \label{eq38d}
  \end{align}
\end{subequations}
Let $\displaystyle c_1=\frac{D_6}{D_3}$, $\displaystyle c_2 = \frac{D_8}{D_4D_2}$ and $\displaystyle c_3 = \frac{D_{12}D_4D_2}{D_8D_6'}$. Then we have
$$
c_2^2=\frac{D_8^2}{D_4^2D_2^2}
\overset{\scriptsize (\ref{eq38a})}{=} \frac{D_L}{D_{12}'D_4^2}
\overset{\scriptsize (\ref{eq38c})}{=}\frac{D_{12}D_{12}''}{D_6'^2D_4^2}
\overset{\scriptsize (\ref{eq38b})}{=}\frac{D_{12}}{D_6'}
\, (=c_2c_3)
\overset{\scriptsize (\ref{eq38d})}{=}\frac{D_6^2}{D_3^2}
\, (=c_1^2),
$$
which implies that $c_1=c_2=c_3$.
\end{proof}

\item $G_T=\left \langle g, h : g^6=h^2=(gh)^2=1 \right \rangle \times \left \langle i \right \rangle \cong D_6 \times C_2$ : Let $L_6=L^{\left \langle hg, g^3i \right \rangle}$, $L_3=L^{\left \langle g^3, hg, i \right \rangle} \subset L_6$, $L_2=L^{\left \langle hg, gi \right \rangle} \subset L_6$ and $L_2'=L^{\left \langle g, hi \right \rangle}$. $T$ is the following type.

\begin{enumerate}[label=(\alph*)]
\addtocounter{enumii}{3}

\item \label{tor24d} $G_T=H_{24, d}:= \left \langle \bigl(\begin{smallmatrix}
1 & 0 & 0\\ 
0 & 0 & -1\\ 
0 & 1 & 1
\end{smallmatrix}\bigr), \bigl(\begin{smallmatrix}
-1 & 0 & 0\\ 
0 & 0 & 1\\ 
0 & 1 & 0
\end{smallmatrix}\bigr), \bigl(\begin{smallmatrix}
-1 & 0 & 0\\ 
0 & -1 & 0\\ 
0 & 0 & -1
\end{smallmatrix}\bigr) \right \rangle$ : $T = T_{L_2'/\Q} \times (T_{L_6/L_3} \cap T_{L_6/L_2})$ and $\displaystyle C(T) = D_2' \frac{D_6}{D_3D_2}$.

\end{enumerate}


\item \label{item3xvi} $G_T=\left \langle g, h : g^4=h^2=(gh)^3=1 \right \rangle \cong S_4$ : Let
$$
L_{12}=L^{ghg^{-1}}, \,
L_{12}'=L^{g^{-1}hg}, \,
L_8=L^{hg}, \,
L_8'=L^{gh}, \,
L_6=L^{\left \langle h, g^2hg^2 \right \rangle}, \,
L_6'=L^{\left \langle ghg^{-1}, g^{-1}hg \right \rangle}, \,
L_6''=L^g
$$
and
$$
L_4=L^{\left \langle hg, ghg^{-1} \right \rangle}, \,
L_4'=L^{\left \langle gh, g^{-1}hg \right \rangle}, \,
L_3=L^{\left \langle h, g^2 \right \rangle}, \,
L_3'=L^{\left \langle g, hg^2h \right \rangle}, \,
L_2=L^{\left \langle g^2, gh \right \rangle}.
$$
Since $L_i'$ is a Galois conjugate of $L_i$ for $i \in \left \{ 3,4,6,8,12 \right \}$, we have $D_i'=D_i$ for these $i$. $T$ is one of the following types. 
    
\begin{enumerate}[label=(\alph*)]
\addtocounter{enumii}{4}

\item \label{tor24e} $G_T=H_{24, e}:= \left \langle \bigl(\begin{smallmatrix}
0 & 0 & 1\\ 
0 & 1 & 0\\ 
-1 & 0 & 0
\end{smallmatrix}\bigr), \bigl(\begin{smallmatrix}
-1 & 0 & 0\\ 
0 & 0 & -1\\ 
0 & -1 & 0
\end{smallmatrix}\bigr) \right \rangle$ : $T = T_{L_6''/L_3'}$ and $\displaystyle C(T) = \frac{D_6''}{D_3}$.

\item \label{tor24f} $G_T=H_{24, f}:= \left \langle \bigl(\begin{smallmatrix}
0 & 0 & -1\\ 
0 & -1 & 0\\ 
1 & 0 & 0
\end{smallmatrix}\bigr), \bigl(\begin{smallmatrix}
1 & 0 & 0\\ 
0 & 0 & 1\\ 
0 & 1 & 0
\end{smallmatrix}\bigr) \right \rangle$ : $T = T_{L_6/L_3}$ and $\displaystyle C(T) = \frac{D_6}{D_3}$.

\item \label{tor24g} $G_T=H_{24, g}:= \left \langle \bigl(\begin{smallmatrix}
0 & -1 & 0\\ 
1 & 1 & 1\\ 
-1 & 0 & 0
\end{smallmatrix}\bigr), \bigl(\begin{smallmatrix}
-1 & -1 & 0\\ 
0 & 1 & 0\\ 
0 & 0 & -1
\end{smallmatrix}\bigr) \right \rangle$ : $T = T_{L_8/L_4} \cap T_{L_8/L_2}$ and $\displaystyle C(T) = \frac{D_8}{D_4D_2}$.

\item \label{tor24h} $G_T=H_{24, h}:= \left \langle \bigl(\begin{smallmatrix}
0 & 1 & 0\\ 
-1 & -1 & -1\\ 
1 & 0 & 0
\end{smallmatrix}\bigr), \bigl(\begin{smallmatrix}
1 & 1 & 0\\ 
0 & -1 & 0\\ 
0 & 0 & 1
\end{smallmatrix}\bigr) \right \rangle$ : $T = T_{L_4/\Q}$ and $\displaystyle C(T) = D_4$.

\item \label{tor24i} $G_T=H_{24, i}:= \left \langle \bigl(\begin{smallmatrix}
1 & 1 & 0\\ 
-2 & -1 & -1\\ 
0 & 0 & 1
\end{smallmatrix}\bigr), \bigl(\begin{smallmatrix}
-1 & -1 & -1\\ 
0 & 0 & 1\\ 
0 & 1 & 0
\end{smallmatrix}\bigr) \right \rangle$ : $T = T_{L_{12}/L_6'} \cap T_{L_{12}/L_4}$ and $\displaystyle C(T) = \frac{D_{12}}{D_6D_4}$.

\item \label{tor24j} $G_T=H_{24, j}:= \left \langle \bigl(\begin{smallmatrix}
-1 & -1 & 0\\ 
2 & 1 & 1\\ 
0 & 0 & -1
\end{smallmatrix}\bigr), \bigl(\begin{smallmatrix}
1 & 1 & 1\\ 
0 & 0 & -1\\ 
0 & -1 & 0
\end{smallmatrix}\bigr) \right \rangle$ : $T = \left \{ v \in T_{L_{12}'/L_6'} : hv \cdot gv \cdot g^2v = 1 \right \}$. Since
$$
1 \rightarrow T \rightarrow T_{L_{12}'/L_6'}
\xrightarrow{v \mapsto hv \cdot gv \cdot g^2v} T_{L_8'/L_4'} \cap T_{L_8'/L_2} \rightarrow 1
$$
is exact, we have $\displaystyle C(T) = \frac{D_{12}D_4D_2}{D_8D_6}$.
\end{enumerate}

\begin{lemma} \label{lem3c}
$\displaystyle \frac{D_6}{D_3} = D_4 = \frac{D_{12}D_4D_2}{D_8D_6}$.
\end{lemma}

\begin{proof}
Let $L_{12}''=L^{g^2}$, $L_{12}'''=L^{g^2hg^2h}$ and $L_6'''=L^{\left \langle g^2, hg^2h \right \rangle}$. Consider the following lattices of subfields. 
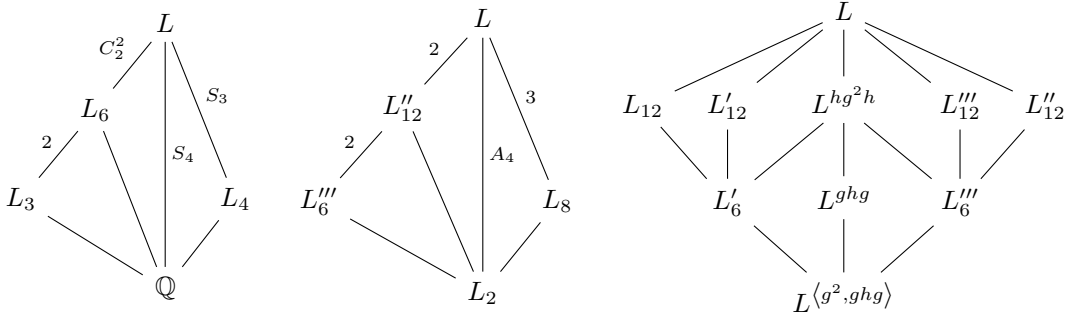
\[ 
\begin{tikzcd}[column sep=0.8em]
&  & L \arrow[ld, "C_2^2"', no head] \arrow[rdd, "S_3", no head] \arrow[ddd, "S_4", no head] & \\
& L_6 \arrow[ld, "2"', no head] \arrow[rdd, no head] &  & \\
L_3 \arrow[rrd, no head] &   &   & L_4 \arrow[ld, no head] \\
&  & \Q  &   
\end{tikzcd}
\quad
\begin{tikzcd}[column sep=0.8em]
&  & L \arrow[ld, "2"', no head] \arrow[rdd, "3", no head] \arrow[ddd, "A_4", no head] & \\
& L_{12}'' \arrow[ld, "2"', no head] \arrow[rdd, no head] &  & \\
L_6''' \arrow[rrd, no head] &   &   & L_8 \arrow[ld, no head] \\
&  & L_2  &   
\end{tikzcd}
\quad
\begin{tikzcd}[column sep=0.8em]
  &  & L \arrow[d, no head] \arrow[rd, no head] \arrow[rrd, no head] \arrow[ld, no head] &  &  \\
L_{12} \arrow[rru, no head] \arrow[rd, no head] & L_{12}' \arrow[d, no head] & L^{hg^2h} \arrow[ld, no head] \arrow[d, no head] \arrow[rd, no head]  & L_{12}''' \arrow[d, no head] & L_{12}'' \arrow[ld, no head] \\
  & L_6' \arrow[rd, no head] & L^{ghg} \arrow[d, no head]   & L_6''' \arrow[ld, no head] &  \\
  &  & L^{\left \langle g^2, ghg \right \rangle}  &   &  
\end{tikzcd}
\]
By the above diagrams, we have the following formulas. Note that $L_6^c=L$ and $\Gal(L/L_6) \cong C_2^2$ so we can apply Proposition \ref{prop22a}(4).
\begin{subequations} 
\begin{align}
   D_6 &= D_3D_4 \label{eq38e} \\
   D_{12}''D_2 &= D_8D_6''' \label{eq38f} \\
   \frac{D_{12}}{D_6} &= \frac{D_{12}''}{D_6'''}. \label{eq38g}
  \end{align}
\end{subequations}
The first two formulas come from the equations (\ref{eq22d}) and (\ref{eq22c2}), and the last formula can be obtained exactly same as the equation (\ref{eq35i}). 
Now the equation (\ref{eq38e}) implies that $\displaystyle \frac{D_6}{D_3} = D_4$ and the equations (\ref{eq38f}) and (\ref{eq38g}) imply that $\displaystyle D_4 = \frac{D_{12}D_4D_2}{D_8D_6}$.
\end{proof}

\begin{lemma} \label{lem3d}
$\displaystyle \frac{D_6''}{D_3}
=\frac{D_8}{D_4D_2}
= \frac{D_{12}}{D_6D_4}$.
\end{lemma}

\begin{proof}
Let $L_6'''' = L^{\left \langle g^2, ghg^2h \right \rangle}$ and consider the following lattices of subfields. Since $(L_6'''')^c=L$ and $\Gal(L/L_6'''') \cong C_2^2$, we can apply Proposition \ref{prop22a}(4). 
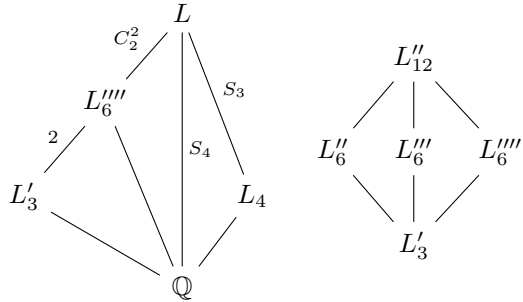
\[ 
\begin{tikzcd}[column sep=0.8em]
&  & L \arrow[ld, "C_2^2"', no head] \arrow[rdd, "S_3", no head] \arrow[ddd, "S_4", no head] & \\
& L_6'''' \arrow[ld, "2"', no head] \arrow[rdd, no head] &  & \\
L_3' \arrow[rrd, no head] &   &   & L_4 \arrow[ld, no head] \\
&  & \Q  &   
\end{tikzcd}
\quad
\begin{tikzcd}[column sep=0.8em]
& L_{12}'' \arrow[ld, no head] \arrow[d, no head] \arrow[rd, no head] & \\
L_6'' \arrow[rd, no head] & L_6''' \arrow[d, no head]  & L_6'''' \arrow[ld, no head] \\
& L_3'  & 
\end{tikzcd}
\]
By the above diagrams, we have the following formulas by the equations (\ref{eq22a}) and (\ref{eq22d}):
\begin{subequations} 
\begin{align}
   D_6'''' &= D_3D_4 \label{eq38h} \\
   D_{12}''D_3^2 &= D_6''D_6'''D_6''''. \label{eq38i}
  \end{align}
\end{subequations}
Now we have
\begin{equation*}
\frac{D_6''}{D_3} 
\overset{\scriptsize (\ref{eq38i})}{=} \frac{D_{12}''D_3}{D_6'''D_6''''}
\overset{\scriptsize (\ref{eq38h})}{=} \frac{D_{12}''}{D_6'''D_4}
\overset{\scriptsize (\ref{eq38g})}{=} \frac{D_{12}}{D_6D_4} \text{ and } \frac{D_{12}''}{D_6'''D_4}
\overset{\scriptsize (\ref{eq38f})}{=} \frac{D_8}{D_4D_2}. \qedhere
\end{equation*}
\end{proof}

\item $G_T=\left \langle g, h : g^4=h^2=(gh)^3=1 \right \rangle \times \left \langle i \right \rangle \cong S_4 \times C_2$ : Let
$$
L_{12}=L^{\left \langle ghg^{-1}, g^{-1}hgi \right \rangle}, \,
L_8=L^{\left \langle hg, ghg^{-1}i \right \rangle}, \,
L_8'=L^{\left \langle gh, g^{-1}hg \right \rangle}, \,
L_6=L^{\left \langle g, hg^2hi \right \rangle}, \,
L_6'=L^{\left \langle ghg^{-1}, g^{-1}hg, i \right \rangle}
$$
and
$$
L_4=L^{\left \langle hg, ghg^{-1}, i \right \rangle}, \,
L_4'=L^{\left \langle gh, g^{-1}hg, i \right \rangle}, \,
L_3=L^{\left \langle g, hg^2h, i \right \rangle}, \,
L_2=L^{\left \langle gi, hi \right \rangle}, \,
L_2'=L^{\left \langle g, h \right \rangle}.
$$
$L_4'$ is a Galois conjugate of $L_4$ so $D_4'=D_4$. $T$ is one of the following types.  

\begin{enumerate}[label=(\alph*)]

    \item \label{tor48a} $G_T = H_{48,a} := \left \langle \bigl(\begin{smallmatrix}
0 & 0 & 1\\ 
0 & 1 & 0\\ 
-1 & 0 & 0
\end{smallmatrix}\bigr), \bigl(\begin{smallmatrix}
-1 & 0 & 0\\ 
0 & 0 & -1\\ 
0 & -1 & 0
\end{smallmatrix}\bigr), \bigl(\begin{smallmatrix}
-1 & 0 & 0\\ 
0 & -1 & 0\\ 
0 & 0 & -1
\end{smallmatrix}\bigr) \right \rangle$ : $T = T_{L_6/L_3}$ and $\displaystyle C(T) = \frac{D_6}{D_3}$.
        
    \item \label{tor48b} $G_T = H_{48,b} := \left \langle \bigl(\begin{smallmatrix}
0 & -1 & 0\\ 
1 & 1 & 1\\ 
-1 & 0 & 0
\end{smallmatrix}\bigr), \bigl(\begin{smallmatrix}
-1 & -1 & 0\\ 
0 & 1 & 0\\ 
0 & 0 & -1
\end{smallmatrix}\bigr), \bigl(\begin{smallmatrix}
-1 & 0 & 0\\ 
0 & -1 & 0\\ 
0 & 0 & -1
\end{smallmatrix}\bigr) \right \rangle$ : $T = T_{L_8/L_4} \cap T_{L_8/L_2}$ and $\displaystyle C(T) = \frac{D_8}{D_4D_2}$.

    \item \label{tor48c} $G_T = H_{48,c} := \left \langle \bigl(\begin{smallmatrix}
1 & 1 & 0\\ 
-2 & -1 & -1\\ 
0 & 0 & 1
\end{smallmatrix}\bigr), \bigl(\begin{smallmatrix}
-1 & -1 & -1\\ 
0 & 0 & 1\\ 
0 & 1 & 0
\end{smallmatrix}\bigr), \bigl(\begin{smallmatrix}
-1 & 0 & 0\\ 
0 & -1 & 0\\ 
0 & 0 & -1
\end{smallmatrix}\bigr) \right \rangle$ : $T = \left \{ v \in T_{L_{12}/L_6'} : hgv \cdot g^2v = hv \right \}$. Since
$$
1 \rightarrow T \rightarrow T_{L_{12}/L_6'}
\xrightarrow{v \mapsto \frac{hgv \cdot g^2v}{hv}} 
T_{L_8'/L_4'} \cap T_{L_8'/L_2'} \rightarrow 1
$$
is exact, we have $\displaystyle C(T) = \frac{D_{12}D_4D_2'}{D_8'D_6'}$.
\end{enumerate}

\begin{lemma} \label{lem3e}
$\displaystyle \frac{D_6}{D_3} = \frac{D_8}{D_4D_2} = \frac{D_{12}D_4D_2'}{D_8'D_6'}$.
\end{lemma}

\begin{proof}
Let
$$
L_{24}=L^i, \, 
L_{12}'=L^{\left \langle g^2, ghg^2hi \right \rangle}, \,
L_6''=L^{\left \langle hg^2h, gi \right \rangle}
\text{ and }
L_6'''=L^{\left \langle g^2, ghg^2h, i \right \rangle}.
$$
Consider the following $3$ lattices of subfields. 
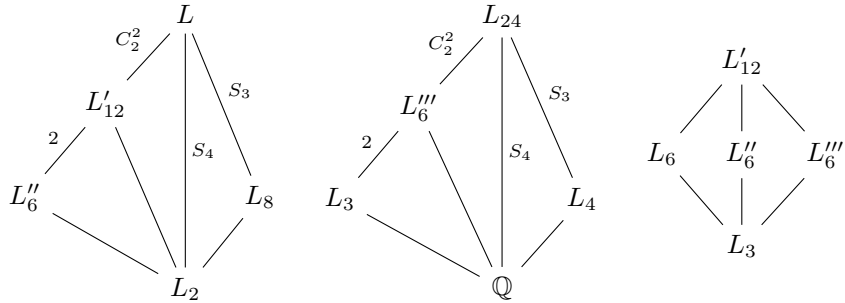
\[ 
\begin{tikzcd}[column sep=0.8em]
&  & L \arrow[ld, "C_2^2"', no head] \arrow[rdd, "S_3", no head] \arrow[ddd, "S_4", no head] & \\
& L_{12}' \arrow[ld, "2"', no head] \arrow[rdd, no head] &  & \\
L_6'' \arrow[rrd, no head] &   &   & L_8 \arrow[ld, no head] \\
&  & L_2  &   
\end{tikzcd}
\quad
\begin{tikzcd}[column sep=0.8em]
&  & L_{24} \arrow[ld, "C_2^2"', no head] \arrow[rdd, "S_3", no head] \arrow[ddd, "S_4", no head] & \\
& L_6''' \arrow[ld, "2"', no head] \arrow[rdd, no head] &  & \\
L_3 \arrow[rrd, no head] &   &   & L_4 \arrow[ld, no head] \\
&  & \Q  &   
\end{tikzcd}
\quad
\begin{tikzcd}[column sep=0.8em]
& L_{12}' \arrow[ld, no head] \arrow[d, no head] \arrow[rd, no head] & \\
L_6 \arrow[rd, no head] & L_6'' \arrow[d, no head]  & L_6''' \arrow[ld, no head] \\
& L_3  & 
\end{tikzcd}
\]
By the above diagrams, we have the following formulas by the equations (\ref{eq22a}) and (\ref{eq22d}): 
\begin{subequations} 
\begin{align}
   D_{12}'D_2 &= D_8D_6'' \label{eq39a} \\
   D_6''' &= D_4D_3 \label{eq39b} \\
   D_{12}'D_3^2 &= D_6D_6''D_6'''. \label{eq39c}
  \end{align}
\end{subequations}
These equations imply that
$$
\displaystyle \frac{D_6}{D_3}
\overset{\scriptsize (\ref{eq39c})}{=} \frac{D_{12}'D_3}{D_6''D_6'''}
\overset{\scriptsize (\ref{eq39a})}{=} \frac{D_8D_3}{D_6'''D_2}
\overset{\scriptsize (\ref{eq39b})}{=} \frac{D_8}{D_4D_2}.
$$
Now we prove the second equality. Let
$$
L_{24}'=L^{hg^2hi}, \,
L_{24}''=L^{hg^2h}, \,
L_{16}=L^{hg}, \,
L_{12}''=L^{\left \langle hg^2h, i \right \rangle}, \, 
L_{12}'''=L^{\left \langle g^2, hg^2h \right \rangle}, \, 
L_{12}''''=L^{\left \langle g^{-1}hg, ghg^{-1}i \right \rangle}
$$
and
$$
L_8''=L^{\left \langle hg, ghg^{-1} \right \rangle}, \,
L_8'''=L^{\left \langle hg, i \right \rangle}, \, 
L_6''''=L^{\left \langle g^2, hg^2h, i \right \rangle}, \, 
L_4''=L^{\left \langle gh, hg \right \rangle}, \, 
L_2''=L^{\left \langle gh, hg, i \right \rangle}.
$$
Since $L_{12}''''$ is a Galois conjugate of $L_{12}$ and $L_8''$ is a Galois conjugate of $L_8'$, we have $D_{12}''''=D_{12}$ and $D_8''=D_8'$. Consider the following $6$ lattices of subfields. 
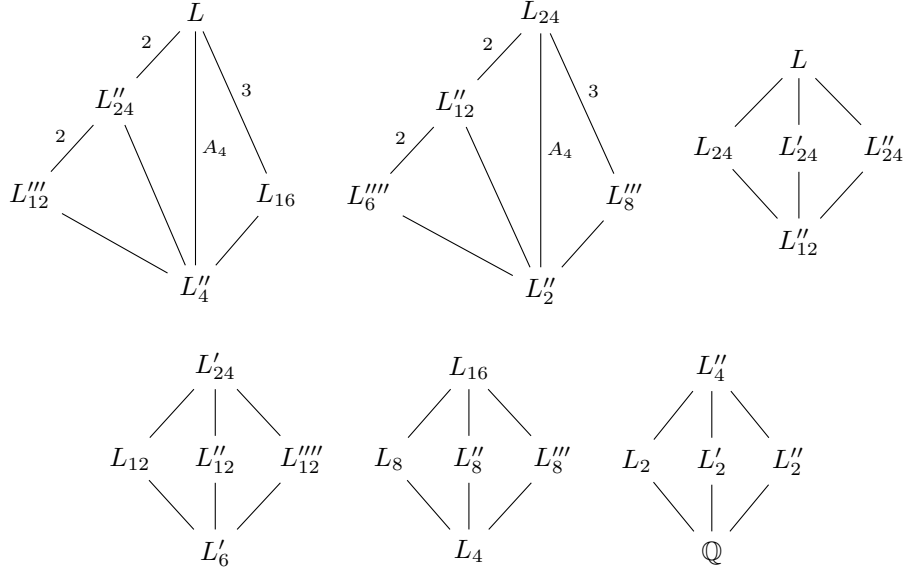
\[ 
\begin{tikzcd}[column sep=0.8em]
&  & L \arrow[ld, "2"', no head] \arrow[rdd, "3", no head] \arrow[ddd, "A_4", no head] & \\
& L_{24}'' \arrow[ld, "2"', no head] \arrow[rdd, no head] &  & \\
L_{12}''' \arrow[rrd, no head] &   &   & L_{16} \arrow[ld, no head] \\
&  & L_4''  &   
\end{tikzcd}
\quad
\begin{tikzcd}[column sep=0.8em]
&  & L_{24} \arrow[ld, "2"', no head] \arrow[rdd, "3", no head] \arrow[ddd, "A_4", no head] & \\
& L_{12}'' \arrow[ld, "2"', no head] \arrow[rdd, no head] &  & \\
L_6'''' \arrow[rrd, no head] &   &   & L_8''' \arrow[ld, no head] \\
&  & L_2''  &   
\end{tikzcd}
\quad
\begin{tikzcd}[column sep=0.8em]
& L \arrow[ld, no head] \arrow[d, no head] \arrow[rd, no head] & \\
L_{24} \arrow[rd, no head] & L_{24}' \arrow[d, no head]  & L_{24}'' \arrow[ld, no head] \\
& L_{12}''  & 
\end{tikzcd}
\]
\[ 
\begin{tikzcd}[column sep=0.8em]
& L_{24}' \arrow[ld, no head] \arrow[d, no head] \arrow[rd, no head] & \\
L_{12} \arrow[rd, no head] & L_{12}'' \arrow[d, no head]  & L_{12}'''' \arrow[ld, no head] \\
& L_6'  & 
\end{tikzcd}
\quad
\begin{tikzcd}[column sep=0.8em]
& L_{16} \arrow[ld, no head] \arrow[d, no head] \arrow[rd, no head] & \\
L_8 \arrow[rd, no head] & L_8'' \arrow[d, no head]  & L_8''' \arrow[ld, no head] \\
& L_4  & 
\end{tikzcd}
\quad
\begin{tikzcd}[column sep=0.8em]
& L_4'' \arrow[ld, no head] \arrow[d, no head] \arrow[rd, no head] & \\
L_2 \arrow[rd, no head] & L_2' \arrow[d, no head]  & L_2'' \arrow[ld, no head] \\
& \Q  & 
\end{tikzcd}
\]
By the above diagrams, we have the following formulas by the equations (\ref{eq22a}) and (\ref{eq22c3}):
\begin{subequations} 
\begin{align}
   D_LD_4''^2 &= D_{24}''D_{16}^2 \label{eq39d} \\
   D_{24}D_2''^2 &= D_{12}''D_8'''^2 \label{eq39e} \\
   D_LD_{12}''^2 &= D_{24}D_{24}'D_{24}'' \label{eq39f} \\
   D_{24}'D_6'^2 &= D_{12}^2D_{12}'' \label{eq39g} \\
   D_{16}D_4^2 &= D_8D_8'D_8''' \label{eq39h} \\
   D_4'' &= D_2D_2'D_2''. \label{eq39i}
  \end{align}
\end{subequations}
These equations imply that
\begin{equation*} 
\begin{split}
\left ( \frac{D_8}{D_4D_2} \right )^2 
&\overset{\scriptsize (\ref{eq39h})}{=} \frac{D_{16}^2D_4^2}{D_8'^2D_8'''^2D_2^2} \\
&\overset{\scriptsize (\ref{eq39d})}{=} \frac{D_LD_4^2D_4''^2}{D_{24}''D_8'^2D_8'''^2D_2^2} \\
&\overset{\scriptsize (\ref{eq39f})}{=} \frac{D_{24}D_{24}'D_4^2D_4''^2}{D_{12}''^2D_8'^2D_8'''^2D_2^2} \\
&\overset{\scriptsize (\ref{eq39e})}{=} \frac{D_{24}'D_4^2D_4''^2}{D_{12}''D_8'^2D_2^2D_2''^2} \\
&\overset{\scriptsize (\ref{eq39g})}{=} \frac{D_{12}^2D_4^2D_4''^2}{D_8'^2D_6'^2D_2^2D_2''^2} \\
&\overset{\scriptsize (\ref{eq39i})}{=} \left ( \frac{D_{12}D_4D_2'}{D_8'D_6'} \right )^2,
\end{split}
\end{equation*}
which proves the second equality. 
\end{proof}

\end{enumerate}

\section{Counting algebraic tori over $\Q$ of dimension $3$} \label{Sec4}

In this section, we count the number of the isomorphism classes of $3$-dimensional tori over $\Q$. We provide asymptotic upper and lower bounds of $N_3^{\tor}(X; H)$ for each finite subgroup $H \neq 1$ of $\GL_3(\Z)$. See Table \ref{table1} for the computation of the numbers $a(H)$ and $b(H)$ for each $H$. This gives an asymptotic upper bound of $N_3^{\tor}(X)$ which is close to the asymptotics predicted by Conjecture \ref{conj1a}. 

\begin{table}[ht]
\centering
\small
\renewcommand{\arraystretch}{0.9}
\begin{tabular}{ |c|c|c|c|c|c|c|c|c|c|c|c| } 
\hline
$\left | H \right |$ & $H$ & $a(H)$ & $b(H)$ & $\left | H \right |$ & $H$ & $a(H)$ & $b(H)$ & $\left | H \right |$ & $H$ & $a(H)$ & $b(H)$ \\ \hline
 \multirow{5}{*}{2} & $H_{2,a}$ & 2 & 1 
 & \multirow{8}{*}{6} & $H_{6,c}$ & 2 & 1 
 & \multirow{9}{*}{12} & $H_{12,c}$ & 1 & 2 \\ \cline{2-4} \cline{6-8} \cline{10-12}
 & $H_{2,b}$ & 1 & 1 
 & & $H_{6,d}$ & 2 & 1 
 & & $H_{12,d}$ & 1 & 2 \\ \cline{2-4} \cline{6-8} \cline{10-12}
 & $H_{2,c}$ & 2& 1 
 & & $H_{6,e}$ & 2 & 2
 & & $H_{12,e}$ & 1 & 2 \\ \cline{2-4} \cline{6-8} \cline{10-12}
 & $H_{2,d}$ & 1 & 1 
 & & $H_{6,f}$ & 1 & 1 
 & & $H_{12,f}$ & 1 & 1 \\ \cline{2-4} \cline{6-8} \cline{10-12}
 & $H_{2,e}$ & 3 & 1 
 & & $H_{6,g}$ & 2 & 2 
 & & $H_{12,g}$ & 1 & 1 \\ \cline{1-4} \cline{6-8}\cline{10-12}
 \multirow{2}{*}{3} & $H_{3,a}$ & 1 & 1 
 & & $H_{6,h}$ & 1 & 1 
 & & $H_{12,h}$ & 1 & 1 \\ \cline{2-4} \cline{6-8}\cline{10-12}
 & $H_{3,b}$ & 1 & 1 
 & & $H_{6,i}$ & 2 & 2 
 & & $H_{12,i}$ & 2 & 2 \\ \cline{1-4} \cline{6-8}\cline{10-12}
 \multirow{15}{*}{4} & $H_{4,a}$ & 2 & 2 
 & & $H_{6,j}$ & 1 & 1 
 & & $H_{12,j}$ & 2 & 2 \\ \cline{2-4} \cline{5-8} \cline{10-12}
 & $H_{4,b}$ & 2 & 1 
 & \multirow{14}{*}{8} & $H_{8,a}$ & 1 & 1 
 & & $H_{12,k}$ & 2 & 2 \\ \cline{2-4} \cline{6-8} \cline{9-12}
 & $H_{4,c}$ & 2 & 2 
 & & $H_{8,b}$ & 1 & 1 
 & \multirow{2}{*}{16} & $H_{16,a}$ & 1 & 3 \\ \cline{2-4} \cline{6-8} \cline{10-12}
 & $H_{4,d}$ & 2 & 1 
 & & $H_{8,c}$ & 1 & 3 
 & & $H_{16,b}$ & 1 & 3 \\ \cline{2-4} \cline{6-8} \cline{9-12}
 & $H_{4,e}$ & 1 & 1 
 & & $H_{8,d}$ & 1 & 3 
 & \multirow{10}{*}{24} & $H_{24,a}$ & 1 & 1 \\ \cline{2-4} \cline{6-8} \cline{10-12}
 & $H_{4,f}$ & 2 & 3 
 & & $H_{8,e}$ & 1 & 3 
 & & $H_{24,b}$ & 1 & 1 \\ \cline{2-4} \cline{6-8} \cline{10-12}
 & $H_{4,g}$ & 1 & 2 
 & & $H_{8,f}$ & 1 & 3 
 & & $H_{24,c}$ & 1 & 1 \\ \cline{2-4} \cline{6-8} \cline{10-12}
 & $H_{4,h}$ & 2 & 3 
 & & $H_{8,g}$ & 2 & 4 
 & & $H_{24,d}$ & 1 & 3 \\ \cline{2-4} \cline{6-8} \cline{10-12}
 & $H_{4,i}$ & 1 & 2 
 & & $H_{8,h}$ & 1 & 2
 & & $H_{24,e}$ & 2 & 4 \\ \cline{2-4} \cline{6-8} \cline{10-12}
 & $H_{4,j}$ & 1 & 2 
 & & $H_{8,i}$ & 1 & 1 
 & & $H_{24,f}$ & 1 & 1 \\ \cline{2-4} \cline{6-8} \cline{10-12}
 & $H_{4,k}$ & 1 & 1 
 & & $H_{8,j}$ & 1 & 1 
 & & $H_{24,g}$ & 2 & 4 \\ \cline{2-4} \cline{6-8} \cline{10-12}
 & $H_{4,l}$ & 2 & 3 
 & & $H_{8,k}$ & 2 & 4 
 & & $H_{24,h}$ & 1 & 1 \\ \cline{2-4} \cline{6-8} \cline{10-12}
 & $H_{4,m}$ & 1 & 2 
 & & $H_{8,l}$ & 1 & 2 
 & & $H_{24,i}$ & 2 & 4 \\ \cline{2-4} \cline{6-8} \cline{10-12}
 & $H_{4,n}$ & 2 & 3 
 & & $H_{8,m}$ & 1 & 1 
 & & $H_{24,j}$ & 1 & 1 \\ \cline{2-4} \cline{6-8} \cline{9-12}
 & $H_{4,o}$ & 1 & 2 
 & & $H_{8,n}$ & 1 & 1 
 & \multirow{3}{*}{48} & $H_{48,a}$ & 1 & 2 \\  \cline{1-4} \cline{5-8}\cline{10-12}
  \multirow{2}{*}{6} & $H_{6,a}$ & 2 & 3 & \multirow{2}{*}{12} & $H_{12,a}$ & 1 & 1 & & $H_{48,b}$ & 1 & 2 \\ \cline{2-4} \cline{6-8}\cline{10-12}
   & $H_{6,b}$ & 1 & 1 & & $H_{12,b}$ & 2 & 5 & & $H_{48,c}$ & 1 & 2 \\ \hline
\end{tabular}
\renewcommand{\arraystretch}{1}
\caption{Computations of $a(H)$ and $b(H)$}
\label{table1}
\end{table}

Throughout this section, $C$ denotes a positive constant which may change from line to line. Denote the set of positive squarefree integers by $\Sqf$. For a prime $p$, denote by $\Sqf_p$ the set of positive integers which are squarefree outside $p$. 
We also introduce a notation which generalizes $N_n(X; G; I)$ defined in Section \ref{Sub21}. Let $J$ be an invariant of pairs of number fields such that for every $X > 0$, there are finitely many pairs of number fields $(K_1, K_2)$ such that $J(K_1, K_2) \leq X$. Denote by $N_{n_1, n_2}^{K_1, K_2}(X; G_1, G_2; J) = N_{n_1, n_2}(X; G_1, G_2; J)$ the number of $(K_1, K_2) \in \NF_{n_1}(G_1) \times \NF_{n_2}(G_2)$ such that $J(K_1, K_2) \leq X$. When $n_i = 2$, we may omit the group $G_i = C_2$ as before.

\subsection{Abelian case} \label{Sub41}

In this section, we prove Conjecture \ref{conj1b} for every finite abelian subgroup $1 \neq H \leq \GL_3(\Z)$. First we prove this except for the cases $H=H_{6, c}$ and $H=H_{6, d}$. 

\begin{proposition} \label{prop41a}
Conjecture \ref{conj1b} holds for every finite abelian subgroup $1 \neq H \leq \GL_3(\Z)$ which is not conjugate to $H_{6, c}$ or $H_{6, d}$.
\end{proposition}

\begin{proof}
By Proposition \ref{prop21a}(1), the conjecture is true if $\left | H \right | \in \left \{ 2, 3 \right \}$ or
$$
H = H_{4, \, x} \,\, (x \in \left \{ b,d,f,h,l,n \right \}).
$$
By Propositions \ref{prop21a}(1) and \ref{prop21c}, the conjecture is true if $H$ is one of
$$
H_{4, \, x} \,\, (x \in \left \{ e,g,i,j,k,m,o \right \}), \, 
H_{6, b} \text{ and }
H_{8, \, x} \,\, (x \in \left \{ c,d,e,f \right \}).
$$
For example, we have
\begin{equation*} 
\begin{split}
N_3^{\tor}(X; H_{4, e})
& = \# \left \{ (L_1, L_3) \in \NF_2 \times \NF_2 : L_1 \neq L_3 \text{ and } D_1D_3^2 \leq X \right \} \\
& = N_{2,2}^{L_1, L_3}(X; D_1D_3^2) - N_2^{L_1}(X; D_1^3) \\
& \sim CX - O(X^{\frac{1}{3}}) \,\, (\because \text{Proposition } \ref{prop21c})\\
& \sim CX.
\end{split}
\end{equation*}
By Propositions \ref{prop21b}(1)(2) and \ref{prop21c}, the conjecture is true if $H$ is one of
$$
H_{4, a}, \, H_{4, c}, \, H_{6, a}, \, 
H_{8, a}, \, H_{8, b} \text{ and } H_{12, a}.
$$
For example, we have
\begin{equation*} 
\begin{split}
N_3^{\tor}(X; H_{8, a})
& = \# \left \{ (L_2, L_4') \in \NF_2 \times \NF_4(C_4) : L_2 \neq L_2' \text{ and } D_2 \frac{D_4'}{D_2'} \leq X \right \} \\
& = N_{2,4}^{L_2, L_4'}(X; C_4; D_2 \frac{D_4'}{D_2'}) - N_4(X; C_4) \\
& =: A_1(X) - A_2(X)
\end{split}
\end{equation*}
where $L_2'$ is the unique quadratic subfield of $L_4'$. Since $N_2(X) \sim CX$ and
$$
N_4^{L_4'}(X; C_4; \frac{D_4'}{D_2'}) \sim CX^{\frac{1}{2}} \log X
$$
by Proposition \ref{prop21b}(1), we have $A_1(X) \sim CX$ by Proposition \ref{prop21c}. Proposition \ref{prop21a}(1) implies that $A_2(X) \sim CX^{\frac{1}{2}}$ so we have $N_3^{\tor}(X; H_{8, a}) \sim CX$.
\end{proof}

Now let $L_6$ be a cyclic sextic field with the cubic subfield $L_3$ and the quadratic subfield $L_2$. Conjecture \ref{conj1b} for the remaining cases ($H_{6,c}$ and $H_{6, d}$) are equivalent to the asymptotic formula
\begin{equation} \label{eq41a}
N_6(X; C_6,\frac{D_6}{D_3}) \sim C X^{\frac{1}{2}}.
\end{equation}
In order to prove this, we need to use analytic methods. We give an explicit formula for the Dirichlet series of $\displaystyle \frac{D_6}{D_3}$ and apply Delange's Tauberian theorem (Proposition \ref{prop23a}) to this. 

\begin{lemma} \label{lem41b}
\begin{equation} \label{eq41b}
\sum_{L_6 \in \NF_6(C_6)} \left ( \frac{D_6}{D_3} \right )^{-s}
= \frac{h(s)g_1(s) - g_2(3s) - g_3(s) + 1}{2}
\end{equation}
where
\begin{equation*}
\begin{split}
h(s) &:= \left ( 1 + \frac{1}{2^{6s}} + \frac{2}{2^{9s}} \right ) 
\left ( 1 + \frac{1}{3^{3s}} + \frac{2}{3^{4s}} + \frac{2}{3^{5s}} \right ), \\ 
g_1(s) &:= \prod_{p \equiv 1 \text{ (mod 6)}}\left ( 1 + \frac{2}{p^{2s}} + \frac{3}{p^{3s}} \right ) 
\prod_{p \equiv -1 \text{ (mod 6)}} \left ( 1+\frac{1}{p^{3s}} \right ), \\  
g_2(s) &:= \left ( 1 + \frac{1}{2^{2s}} + \frac{2}{2^{3s}} \right ) 
\prod_{p \text{ odd}} \left ( 1 + \frac{1}{p^{s}} \right ), \\  
g_3(s) &:= \left ( 1 + \frac{2}{3^{4s}} \right )
\prod_{p \equiv 1 \text{ (mod 6)}} \left ( 1 + \frac{2}{p^{2s}} \right ).
\end{split}
\end{equation*}
\end{lemma}

\begin{proof}
Denote the conductor of an abelian number field $M$ by $\cond (M)$. By the conductor-discriminant formula, we have $\displaystyle \cond(L_6) = \left ( \frac{D_6}{D_2D_3} \right )^{\frac{1}{2}}$. We also have 
$$
\cond(L_6) = \lcm(\cond(L_2), \cond(L_3)) = \lcm(D_2, D_3^{\frac{1}{2}})
$$ 
(\cite[(4)]{Mak93}) so
$$
\frac{D_6}{D_3} = \lcm(D_2^3, D_2D_3).
$$
The Dirichlet series of quadratic fields and cyclic cubic fields are given by
$$
\sum_{L_2 \in \NF_2} \frac{1}{D_2^s} = g_2(s) - 1
$$
(\cite[Section 2]{CDO06}) and
$$
\sum_{L_3 \in \NF_3(C_3)} \frac{1}{D_3^s} = \frac{g_3(s)-1}{2} 
$$
(\cite[Section 3]{CDO06}). Let $S_1$ and $S_2$ be multisets of positive integers defined by
\begin{equation*} 
\begin{split}
S_1 & := \left \{ D_2 : L_2 \in \NF_2 \right \} \cup \left \{ 1 \right \} \\
S_2 & := \left \{ D_3 : L_3 \in \NF_3(C_3) \right \} \cup \left \{ D_3 : L_3 \in \NF_3(C_3) \right \} \cup \left \{ 1 \right \}
\end{split}
\end{equation*}
(union as multisets). Then we have $\displaystyle \sum_{n_1 \in S_1} \frac{1}{n_1^s} = g_2(s)$ and $\displaystyle \sum_{n_2 \in S_2} \frac{1}{n_2^s} = g_3(s)$ so
\begin{equation*}
\begin{split}
    \sum_{L_6 \in \NF_6(C_6)} \left ( \frac{D_6}{D_3} \right )^{-s}
& = \frac{1}{2} \sum_{\substack{(n_1, n_2) \in S_1 \times S_2 \\ n_1 \neq 1, \, n_2 \neq 1}} \frac{1}{\lcm(n_1^3, n_1n_2)^s} \\
& = \frac{1}{2} \sum_{\substack{(n_1, n_2) \in S_1 \times S_2}} \frac{1}{\lcm(n_1^3, n_1n_2)^s}
- \frac{1}{2} \sum_{n_1 \in S_1} \frac{1}{n_1^{3s}}
- \frac{1}{2} \sum_{n_2 \in S_2} \frac{1}{n_2^{s}}
+ \frac{1}{2} \\
& = \frac{h(s)g_1(s)}{2} - \frac{g_2(3s)}{2} - \frac{g_3(s)}{2} + \frac{1}{2}. \qedhere
\end{split}
\end{equation*}
\end{proof}

\begin{proposition} \label{prop41c}
Conjecture \ref{conj1b} holds for $H=H_{6, c}$ and $H=H_{6, d}$. 
\end{proposition}

\begin{proof}
By \cite[Theorem 1.1]{Wri89}, $\displaystyle g_2(s) \in \mathcal{M} (1, \frac{1}{2})$ and $\displaystyle g_3(s) \in \mathcal{M} (\frac{1}{2}, \frac{1}{3})$. (See Section \ref{Sub23} for the definition of $\mathcal{M} (\alpha, \beta)$ for $\alpha > \beta > 0$.) Since 
\begin{equation*}
\begin{split}
& \frac{h(s)g_1(s)-g_3(s)}{g_3(s)} \\
= &
\left ( 1 + \frac{1}{2^{6s}} + \frac{2}{2^{9s}} \right ) 
\left ( 1 + \frac{3^{2s}+2}{3^{5s}+2 \cdot 3^s} \right )
\prod_{p \equiv 1 \text{ (mod 6)}} \left ( 1 + \frac{3}{p^{3s}+2p^s} \right )
\prod_{p \equiv -1 \text{ (mod 6)}} \left ( 1+\frac{1}{p^{3s}} \right ) - 1
\end{split}
\end{equation*}
converges absolutely for $\displaystyle \re(s) > \frac{1}{3}$ and 
$$
\lim_{s \rightarrow \frac{1}{2}} \frac{h(s)g_1(s)-g_3(s)}{g_3(s)} > 0, 
$$
we have $\displaystyle h(s)g_1(s)-g_3(s) \in \mathcal{M} (\frac{1}{2}, \frac{1}{3})$. Now Lemma \ref{lem41b} implies that
$$
\sum_{L_6 \in \NF_6(C_6)} \left ( \frac{D_6}{D_3} \right )^{-s} 
\in \mathcal{M} (\frac{1}{2}, \frac{1}{3})
$$
so Proposition \ref{prop23a} finishes the proof.
\end{proof}

\subsection{Lower bound for non-abelian case} \label{Sub42}

There are $39$ conjugacy classes of finite non-abelian subgroups of $\GL_3(\Z)$. By the computation of the numbers $C(T)$ in Section \ref{Sec3} and Table \ref{table1}, it is enough to consider the asymptotics of $N_3^{\tor}(X; H)$ when $H$ is one of the following $16$ subgroups of $\GL_3(\Z)$:
\begin{itemize}
\item $H_{6, e}$, $H_{6, f}$ ($S_3$)
\item $H_{8, x} \,\, (x \in \left \{ g, h, i \right \})$ ($D_4$), $H_{16, a}$ ($D_4 \times C_2$)
\item $H_{12, i}$ ($A_4$), $H_{24, b}$ ($A_4 \times C_2$)
\item $H_{12,x} \,\, (x \in \left \{ b, c, d, f \right \})$ ($D_6$), $H_{24, d}$ ($D_6 \times C_2$)
\item $H_{24, g}$, $H_{24, h}$ ($S_4$), $H_{48, b}$ ($S_4 \times C_2$)
\end{itemize}

\begin{proposition} \label{prop42a}
Conjecture \ref{conj1b} holds for $H= H_{6,f}, \, H_{8, h}, \, H_{8, i}, \, H_{16, a} , \, H_{12, d}$ and $H_{24, h}$. Under the assumption of Malle's conjecture for quartic $A_4$-fields, it also holds for $H=H_{12, i}$. 
\end{proposition}

\begin{proof}
The cases $H_{6,f}$, $H_{8, i}$ and $H_{24, h}$ follow from Proposition \ref{prop21a}(2), the case $H_{8, h}$ follows from Proposition \ref{prop21b}(3) and the case $H_{12, d}$ follows from Propositions \ref{prop21a}(2) and \ref{prop21c}. 
When $G_T=H_{16, a}$, we have $\displaystyle C(T)=D_2'' \frac{D_4'}{D_2'}$ for $L_2'' \in \NF_2$ and $L_4' \in \NF_4(D_4)$ (with the unique quadratic subfield $L_2'$) such that $L_2'' \cap L_4'^c = \Q$. Therefore
\begin{equation*} 
\begin{split}
N_3^{\tor}(X; H_{16, a})
& = \# \left \{ (L_2'', L_4') \in \NF_2 \times \NF_4(D_4) : L_2'' \cap L_4'^c = \Q \text{ and } D_2'' \frac{D_4'}{D_2'} \leq X \right \} \\
& = N_{2,4}^{L_2'', L_4'}(X; D_4; D_2'' \frac{D_4'}{D_2'})  - \# \left \{ (L_2'', L_4') \in \NF_2 \times \NF_4(D_4) : L_2'' \subset L_4'^c \text{ and } D_2'' \frac{D_4'}{D_2'} \leq X \right \} \\
& \sim CX (\log X)^2 - O(X \log X) \\
& \sim CX (\log X)^2
\end{split}
\end{equation*}
by Propositions \ref{prop21b}(3) and \ref{prop21c}. The second statement is trivial because $N_3^{\tor}(X; H_{12, i}) = N_4(X; A_4)$.
\end{proof}

\begin{remark} \label{rmk42b}
Unconditionally, we have
$$
X^{\frac{1}{2}} \ll N_4(X; A_4) \ll_{\varepsilon} X^{\frac{1}{2}+\gamma +\varepsilon}
$$
for $\gamma = 0.2784...$ by \cite[Theorem 3]{Bai80} and \cite[Theorem 1.4]{BSTTTZ20}. 
\end{remark}

Let $\WW$ be the set of the $9$ remaining cases, i.e.
$$
\WW := \left \{ H_{6,e}, \, H_{8,g}, \, H_{12, b}, \, H_{12, c}, \, H_{12, f}, \, H_{24, b}, \, H_{24,d}, \, H_{24, g}, \, H_{48, b} \right \}.
$$
We provide upper and lower bounds of $N_3^{\tor}(X; H)$ for each $H$ in $\WW$, which are summarized in Table \ref{table2} below. The upper bounds for $H = H_{12, b}, \, H_{12, c}, \, H_{24, d}$ are under the assumption of Conjecture \ref{conj1d} and the upper bound for $H_{24, g}$ is under the assumption of Malle's conjecture for the Galois group $12T8$. The other upper and lower bounds are unconditional. 

\begin{table}[ht]
\centering
\small
\begin{tabular}{ |c|c|c|c| } 
\hline
$H$ & Malle & Lower & Upper  \\ \hline
 $H_{6, e}$, $H_{6, g}$, $H_{6, i}$ & $X^{\frac{1}{2}} \log X$ & $X^{\frac{1}{2}} $ & $X^{\frac{1}{2}} (\log X)^2$ \\ \hline
 $H_{8, g}$, $H_{8, k}$ & $X^{\frac{1}{2}} (\log X)^3$ & $X^{\frac{1}{2}} (\log X)^2$ & $X^{\frac{3}{4} - \delta}$ ($\delta > 0$) \\ \hline
 $H_{12, b}$ & $X^{\frac{1}{2}} (\log X)^4$ & $X^{\frac{1}{2}} \log X$ & $X^{\frac{1}{2}} (\log X)^{6+\varepsilon}$ \\ \hline
 $H_{12, c}$ & $X \log X$ & $X$ & $X (\log X)^{1+\varepsilon}$ \\ \hline
 $H_{12, f}$, $H_{12, g}$, $H_{12, h}$ & $X$ & $X$ & $X$ \\ \hline
 $H_{24, a}$, $H_{24, b}$, $H_{24, c}$ & $X$ & $X$ & $X (\log X)^3 \log \log X$ \\ \hline
 $H_{24, d}$ & $X (\log X)^2$ & $X \log X$ & $X (\log X)^{2+\varepsilon}$ \\ \hline
 $H_{24, e}$, $H_{24, g}$, $H_{24, i}$ & $X^{\frac{1}{2}} (\log X)^3$ & $X^{\frac{1}{2}}$  & $X^{\frac{9}{10}}$ \\ \hline
 $H_{48, a}$, $H_{48, b}$, $H_{48, c}$ & $X \log X$ & $X$ & $X (\log X)^4 \log \log X$ \\ \hline
\end{tabular}
\caption{Upper and lower bounds of $N_3^{\tor}(X;H)$}
\label{table2}
\end{table}

We concentrate on the lower bounds in this section. The result is satisfactory, in the sense that the asymptotic inequality $X^{\frac{1}{a(H)}} \ll N_3^{\tor}(X; H)$ holds for every finite subgroup $H \neq 1$ of $\GL_3(\Z)$. By Propositions \ref{prop41a}, \ref{prop41c}, \ref{prop42a} and Remark \ref{rmk42b}, it is enough to prove this for $H \in \WW$. The case $H=H_{24, b}$ is considered separately because its proof relies on the work of Cohen and Thorne \cite{CT16}.

For a cyclic cubic field $k$, denote by $\mathcal{F}(k)$ the set of quartic $A_4$-fields whose cubic resolvent is $k$. For every $K \in \mathcal{F}(k)$, there is $f(K) \in \Sqf_2$ which satisfies $D_K = D_k f(K)^2$. For every $f \in \Sqf_2$, denote
$$
\mathcal{F}(k, f^2) := \left \{ K \in \mathcal{F}(k) : D_K = D_k f^2 \right \}.
$$

\begin{proposition} \label{prop42c}
$N_3^{\tor}(X; H_{24, b}) \gg X$.
\end{proposition}

\begin{proof}
Let $E$ be the maximal real subfield of $\Q (\zeta_7)$. Then $E \in \NF_3(C_3)$, $D_{E}=49$ and $h_{E}=1$. A prime $p \neq 7$ splits completely in $E$ if and only if $p$ is a cubic residue modulo $7$, i.e. $p \equiv \pm 1 \,\, ( \text{mod } 7)$. Therefore
    \begin{equation*}
    \frac{1}{3} + \sum_{K \in \mathcal{F}(E)}\frac{1}{f(K)^s} = \frac{1}{3}\left ( 1 + \frac{3}{2^{3s}} \right ) \prod_{p \equiv \pm 1 \,\, (\text{mod } 7) } \left ( 1+\frac{3}{p^s} \right )
    \end{equation*}
    by \cite[Theorem 1.4]{CT16}. Ignoring the even parts, we obtain
    \begin{equation} \label{eq42b}
    \begin{split}
    \sum_{\gcd(f, 14)=1} \frac{\left | \mathcal{F}(E, f^2) \right |}{f^s} 
    & = \sum_{\substack{K \in \mathcal{F}(E) \\ \gcd(f(K), 14)=1}} \frac{1}{f(K)^s} \\
    & = \frac{1}{3} \prod_{p \equiv \pm 1 \,\, (\text{mod } 7) } \left ( 1+\frac{3}{p^s} \right ) - \frac{1}{3} \\
    & \in \mathcal{M}(1, \frac{1}{2})
    \end{split}
    \end{equation}
    by Lemma \ref{lem23b}. 
 
Let $f$ be a positive squarefree integer which is coprime to $14$ and define $L_{2, f} := \Q(\sqrt{f})$ (so $D_{2, f} \in \left \{ f, 4f \right \}$). Let $L_{4, f}$ be an element of $\mathcal{F}(E, f^2)$ and $L_{8, f} := L_{4, f}L_{2, f} \in \NF_8(A_4 \times C_2)$. 
    For any prime $p$ dividing $f$, the splitting type of $p$ in $L_{4, f}$ is $(1^21^2)$ or $(2^2)$ by \cite[Theorem 5.1]{CT16} so $v_p(D_{8, f})=4$ by Proposition \ref{prop22b}. Also $D_{8, f} \mid D_{4, f}^2 D_{2, f}^4$ so $\displaystyle v_7 \left ( \frac{D_{8, f}}{D_{4, f}D_{2, f}} \right ) \leq 2$ and $\displaystyle v_2 \left ( \frac{D_{8, f}}{D_{4, f}D_{2, f}} \right ) \leq 6$. These imply that
    \begin{equation} \label{eq42a}
    \frac{D_{8, f}}{D_{4, f}D_{2, f}} \leq cf
    \end{equation}
    for $c=56^2$. Now we have
    \begin{equation*}
        N_3^{\tor}(X; H_{24, b})
        = N_{4,2}^{L_4, L_2}(X; A_4; \frac{D_8}{D_4D_2})
        \geq \sum_{\substack{f \leq \frac{X}{c} \\ \gcd(f, 14)=1}} \left | \mathcal{F}(E, f^2) \right |
        \sim CX
    \end{equation*}
    by the relation (\ref{eq42b}), the inequality (\ref{eq42a}) and Proposition \ref{prop23a}.
\end{proof}

\begin{lemma} \label{lem42d}
Let $L_2$, $L_4$ and $L_8$ be number fields as in \ref{item3xvi} of Section \ref{Sec3}. Then
$$
v_p(D_8) \leq v_p(D_4^3D_2)
$$
for every odd prime $p$. Here $v_p(m)$ denotes the exponent of $p$ in $m$.
\end{lemma}

\begin{proof}
Assume that $v_p(D_8) > v_p(D_4^3D_2)$ for an odd prime $p$.
\begin{itemize}
    \item $L_8=L_4L_2$ implies that $v_p(D_8) \leq 2v_p(D_4)+4v_p(D_2)$ so $v_p(D_2)=1$. We also have
    $$
    3v_p(D_4)+1 < v_p(D_8) \leq 2v_p(D_4)+4
    $$
    so $v_p(D_4) < 3$.
    
    \item The cubic resolvent of $L_4 \in \NF_4(S_4)$ is $L_3 \in \NF_3(S_3)$ and the quadratic resolvent of $L_3$ is $L_2$. Therefore $v_p(D_4)$, $v_p(D_3)$ and $v_p(D_2)$ have the same parity so $v_p(D_3)=v_p(D_4)=1$ and $v_p(D_6)=v_p(D_3D_4)=2$. 
    
    \item Let $c := v_p(D_8)$. Then $v_p(D_6'')=c-1$ and $v_p(D_{12})=c+1$ by Lemma \ref{lem3d}. Since $L/L_4$ is a Galois extension with a Galois group $S_3$, we have $D_LD_4^2=D_{12}^2D_8$ by Proposition \ref{prop22a}(2) so $v_p(D_L)=3c$. 
    
    \item The relation $v_p(D_L) \geq 4v_p(D_6'')$ implies that $c \leq 4$, which contradicts the assumption. This finishes the proof. \qedhere
\end{itemize}
\end{proof}

\begin{theorem} \label{thm42e}
For every $H \in \WW$, the lower bound of $N_3^{\tor}(X; H)$ is given as in Table \ref{table2}. In particular we have
$$
X^{\frac{1}{a(H)}} \ll N_3^{\tor}(X; H)
$$
for every finite nontrivial subgroup $H$ of $\GL_3(\Z)$.
\end{theorem}

\begin{proof}
For each $H \in \WW$, we use the notation as in Section \ref{Sec3}.
\begin{itemize}
    \item $H=H_{6, e}$ : Since $L_2$ is the quadratic resolvent of $L_3 \in \NF_3(S_3)$, we have $D_2 \leq D_3$ so
    \begin{equation*}
            N_3^{\tor}(X; H_{6, e}) = N_3^{L_3}(X; S_3; D_2D_3) \geq N_3^{L_3}(X; S_3; D_3^2) \sim CX^{\frac{1}{2}}.
    \end{equation*}
    
    \item $H=H_{8, g}$ : Since $\displaystyle \left ( \frac{D_K D_{M_2}}{D_{K_2}} \right )^2 \leq D_M \left ( \frac{D_{M_2}}{D_{K_2}} \right )^2 = D_L$ by the equation (\ref{eq35g}), we have
    \begin{equation*}
        N_3^{\tor}(X; H_{8, g})
        = N_8^{L}(X; D_4; \frac{D_K D_{M_2}}{D_{K_2}})
        \geq N_8^{L}(X^2; D_4; D_L) \sim CX^{\frac{1}{2}} (\log X)^2
    \end{equation*}
    by Proposition \ref{prop21a}(4).
    
    \item $H=H_{12, b}$ : Since $L_7$ is the quadratic resolvent of $L_4 \in \NF_3(S_3)$, we have $D_7 \leq D_4$ so
        \begin{equation*}
        \begin{split}
            N_3^{\tor}(X; H_{12, b})
            & = N_6^{L_1}(X; D_6; \frac{D_1D_7}{D_4D_6}) \\
            & \geq N_6^{L_1}(X; D_6; \frac{D_1}{D_6}) \\
            & \geq \# \left \{ (L_4, L_6) \in \NF_3(S_3) \times \NF_2 :  D_4^2D_6^2 \leq X \text{ and } L_6 \not\subset L_4^c \right \} \\
            & \sim CX^{\frac{1}{2}} \log X - O(X^{\frac{1}{2}}) \\
            & \sim CX^{\frac{1}{2}} \log X.
        \end{split}
    \end{equation*}
        
    \item $H \in \left \{ H_{12, c}, \, H_{12, f}, \, H_{48, b} \right \}$ : They can be proved as in the proof of \cite[Theorem 4.9]{Lee21}.

    \item $H=H_{24, b}$ : See Proposition \ref{prop42c}.
    
    \item $H=H_{24, d}$ : Since $D_6$ has $3$ subgroups of order $6$, $L_6^c$ has $3$ quadratic subfields. For a constant $M>0$ which satisfies $\# \left \{ L_2' \in \NF_2 : D_2' \leq M \right \} = N_2(M) \geq 6$, we have
    \begin{equation*}
    \begin{split}
        N_3^{\tor}(X; H_{24, d})
        & = \# \left \{ (L_6, L_2') \in \NF_6(D_6) \times \NF_2 :  D_2' \frac{D_6}{D_3D_2} \leq X \text{ and } L_2' \not\subset L_6^c \right \} \\
        & \geq \sum_{\substack{L_6 \in \NF_6(D_6) \\ \frac{D_6}{D_3D_2} \leq \frac{X}{M}}} \left ( 
        N_2(X \frac{D_3D_2}{D_6} ) - 3 \right ) \\
        & \geq \frac{1}{2} \sum_{\substack{L_6 \in \NF_6(D_6) \\ \frac{D_6}{D_3D_2} \leq \frac{X}{M}}} N_2(X \frac{D_3D_2}{D_6} ) \\
        & \gg \frac{X}{M} \log \frac{X}{M}
        \end{split}
    \end{equation*}
    by Propositions \ref{prop1c}(2) and \ref{prop21c}. 
        
    \item $H=H_{24, g}$ : By Lemma \ref{lem42d}, 
    \begin{equation*}
        N_3^{\tor}(X; H_{24, g})
        = N_4^{L_4}(X; S_4; \frac{D_8}{D_4D_2})
        \geq N_4^{L_4}(CX; S_4; D_4^2) 
        \sim CX^{\frac{1}{2}}. \qedhere
    \end{equation*}
\end{itemize}
\end{proof}

\subsection{Upper bound for non-abelian case} \label{Sub43}
Now we move to the upper bounds. If $H$ is one of $H_{12, b}$, $H_{12, c}$ and $H_{24, d}$, the Cohen-Lenstra heuristics for $p=3$ improves the upper bound of $N_3^{\tor}(X; H)$. 

\begin{proposition} \label{prop43a}
\begin{enumerate}
    \item $\displaystyle N_3^{\tor}(X; H) \ll_{\varepsilon} X^{a(H)\left ( 1 + \frac{\log 2 + \varepsilon}{\log \log X} \right )}$ for $H \in \left \{ H_{12, b}, H_{12, c}, H_{24, d} \right \}$.
    
    \item Under the assumption of Conjecture \ref{conj1d}, the upper bounds of $N_3^{\tor}(X; H)$ for $H \in \left \{ H_{12, b}, H_{12, c}, H_{24, d} \right \}$ are given as in Table \ref{table2}.
\end{enumerate}
\end{proposition}

\begin{proof}
The case $H=H_{12, c}$ follows from the equality $N_3^{\tor}(X; H_{12, c}) = N_2^{\tor}(X; H_{12, A})$ and Proposition \ref{prop1c}. The function $\displaystyle f(X) = \frac{\log X}{\log \log X}$ is increasing on $X > e^e = 15.15...$ so
    $$
    N_2^{\tor}(Y; H_{12, A}) \ll_{\varepsilon} Y X^{\frac{\log 2 + \varepsilon}{\log \log X}}
    $$
    for $X \geq Y \geq 16$. Since $N_3^{\tor}(X; H_{24, d})$ is bounded above by the product distribution of $N_2^{\tor}(X; H_{12, A})$ and $N_2(X)$, we have
\begin{equation*}
\begin{split}
N_3^{\tor}(X; H_{24, d})
& \leq 2 \sum_{n \leq X} N_2^{\tor}(\frac{X}{n}; H_{12, A}) \\
& \ll_{\varepsilon} \sum_{n \leq \frac{X}{16}} \frac{X}{n} X^{\frac{\log 2 + \varepsilon}{\log \log X}} + \sum_{\frac{X}{16} < n \leq X} N_2^{\tor}(16; H_{12, A}) \\
& \ll_{\varepsilon} X^{1 + \frac{\log 2 + \varepsilon}{\log \log X}}.
\end{split}
\end{equation*}
If we assume Conjecture \ref{conj1d}, then we have
\begin{equation*}
\begin{split}
N_3^{\tor}(X; H_{24, d})
& \leq 2 \sum_{n \leq X} N_2^{\tor}(\frac{X}{n}; H_{12, A}) \\
& \ll_{\varepsilon} \sum_{n \leq X} \frac{X}{n} (\log X)^{1 + \varepsilon}\\
& \ll_{\varepsilon} X (\log X)^{2 + \varepsilon}.
\end{split}
\end{equation*}

Now consider the case $H=H_{12, b}$. Let $F$, $K$, $L$, $E$, $f$, $m_1$, $m_2$, $w(f)$ and $g(X)$ be defined as in \cite[Section 4.2]{Lee21}. Denote
$$
A_1(X) := N_{3,2}^{F,K}(X; S_3; D_E \frac{D_F D_K^2}{m^2})
$$
and
$$
A_2(X; B) := \sum_{\substack{m_1, m_2 \in \, \, \Sqf \\ (m_1m_2, 6)=1 \\ (m_1, m_2)=1}}
\sum_{\substack{E \in \NF_2 \\ m_1 \mid D_E}} 
\sum_{\substack{f \in \, \Sqf_3 \\ m_2 \mid f \\  D_E f^2  \leq X \\ B \leq D_E f < 2B}}
\frac{X^{\frac{1}{2}}}{D_E f}h_3(E) \cdot 2^{w(f)}
$$
for $B \leq X^{\frac{1}{2}}$. Following the arguments of \cite[Section 4]{Lee21}, one can show that
\begin{equation} \label{eq43a1}
N_3^{\tor}(X; H_{12, b}) \leq A_1(\beta X)
\end{equation}
for $\beta := 2^9 3^3$, 
\begin{equation} \label{eq43a2}
A_1(X) \ll \sum_{i=0}^{\left \lfloor \log_2 X^{\frac{1}{2}} \right \rfloor} A_2(X; 2^i)
\end{equation}
and
\begin{equation} \label{eq43a3}
\begin{split}
A_2(X; B) 
& \ll \frac{X^{\frac{1}{2}}}{B} 
\sum_{\substack{f \in \, \Sqf_3 \\ f < 2B }} 2^{w(f)} \tau(f)
g(\frac{2B}{f}) \\
& \ll_{\varepsilon} \frac{X^{\frac{1}{2}}}{B} 
\sum_{\substack{f \in \, \Sqf_3 \\ f < 2B }} \tau(f)^2
\frac{2B}{f} X^{\frac{1}{2} \left ( \frac{\log 2 + \varepsilon}{\log \log X} \right )} \;\; \text{(\cite[Corollary 4.4]{Lee21})} \\
& \ll_{\varepsilon} X^{\frac{1}{2}\left ( 1 + \frac{\log 2 + \varepsilon}{\log \log X} \right )} \sum_{f<2B} \frac{\tau(f)^2}{f}
\end{split}
\end{equation}
for $B \leq X^{\frac{1}{2}}$. By summing over the intervals $f \in [2^j, 2^{j+1})$ for $0 \leq j \leq \log_2 (2B)$, we have
\begin{equation} \label{eq43a5}
\begin{split}
\sum_{f<2B} \frac{\tau(f)^2}{f}
& \leq \sum_{j=0}^{\left \lfloor \log_2 (2B) \right \rfloor} \frac{1}{2^j} \sum_{f < 2^{j+1}} \tau(f)^2 \\
& \ll \sum_{j=0}^{\left \lfloor \log_2 (2B) \right \rfloor} \frac{1}{2^j} 2^{j+1} (\log 2^{j+1})^3 \;\; \text{(Lemma \ref{lem23c})} \\
& \ll (\log B)^4.
\end{split}
\end{equation}
The inequalities (\ref{eq43a1}), (\ref{eq43a2}), (\ref{eq43a3}) and (\ref{eq43a5}) imply that $N_3^{\tor}(X; H_{12, b}) \ll_{\varepsilon} X^{\frac{1}{2}\left ( 1 + \frac{\log 2 + \varepsilon}{\log \log X} \right )}$. 

If we assume Conjecture \ref{conj1d}, then
\begin{equation} \label{eq43a4}
\begin{split}
A_2(X; B) 
& \ll_{\varepsilon} \frac{X^{\frac{1}{2}}}{B} 
\sum_{\substack{f \in \, \Sqf_3 \\ f < 2B }} 2^{w(f)} \tau(f) \cdot \frac{2B}{f} (\log B)^{1+\varepsilon} \;\; \text{(\cite[Corollary 4.8]{Lee21})} \\
& \ll_{\varepsilon} X^{\frac{1}{2}} (\log B)^{1+\varepsilon}
\sum_{f<2B} \frac{\tau(f)^2}{f}.
\end{split}
\end{equation}
The inequalities (\ref{eq43a1}), (\ref{eq43a2}), (\ref{eq43a5}) and (\ref{eq43a4}) imply that $N_3^{\tor}(X; H_{12, b}) \ll_{\varepsilon} X^{\frac{1}{2}} (\log X)^{6+\varepsilon}$.
\end{proof}

We proceed by case-by-case analysis. 

\begin{proposition} \label{prop43b}
The following estimates hold:
\begin{enumerate}
    \item \label{43b1} $N_3^{\tor}(X; H_{6, e}) \ll X^{\frac{1}{2}}(\log X)^2$.
    
    \item \label{43b2} $N_3^{\tor}(X; H_{12, f}) \ll X$.
    
    \item \label{43b3} $N_3^{\tor}(X; H_{48, b}) \ll X(\log X)^4 \log \log X$.
    
    \item \label{43b4} $N_3^{\tor}(X; H_{24, b}) \ll X(\log X)^3 \log \log X$.
\end{enumerate}
\end{proposition}

\begin{proof}
\begin{enumerate}
    \item By \cite[Theorem 2.5]{CT14}, the number of $F \in \NF_3(S_3)$ such that $D_F=D_Ef^2$ for given $E \in \NF_2$ and $f \in \Sqf_3$ is bounded by $O(h_3(E) \cdot 2^{w(f)})$. This implies that
    \begin{equation*}
\begin{split}
N_3^{\tor}(X; H_{6, e})
& = N_3^{F}(X; S_3; D_FD_E) \\
& \ll \sum_{\substack{f \in \, \Sqf_3 \\ f < X^{\frac{1}{2}} }} \sum_{\substack{E \in \NF_2 \\ D_E \leq \frac{X^{\frac{1}{2}}}{f}}} h_3(E) 2^{w(f)} \\
& \ll \sum_{\substack{f < X^{\frac{1}{2}} }} \frac{X^{\frac{1}{2}}}{f} \tau(f) \;\; \text{(\cite[Theorem 3]{DH71})} \\
& \ll X^{\frac{1}{2}}(\log X)^2.
\end{split}
\end{equation*}
The last inequality can be deduced as in the inequality (\ref{eq43a5}) above.

\item By Proposition \ref{prop21a}(3), we have
\begin{equation*} 
N_3^{\tor}(X; H_{12, f}) 
= N_6^{L_1}(X; D_6; \frac{D_1}{D_4})
\leq N_6^{L_1}(X^2; D_6; D_1) 
\sim CX.
\end{equation*}

\item Let $F \in \NF_4(S_4)$, $K \in \NF_2$ such that $K \cap F^c = \Q$ and $L = FK \in \NF_8(S_4 \times C_2)$. Also let
\begin{equation*}
\begin{split}
S_1 & := \left \{ p >3 : p \mid \gcd (D_F, D_K) \text{ and } p \text{ has splitting type } (1^211), (1^22) \text{ or } (1^31) \text{ in } F \right \} \\
S_2 & := \left \{ p >3 : p \mid \gcd (D_F, D_K) \text{ and } p \text{ has splitting type } (1^21^2), (2^2) \text{ or } (1^4) \text{ in } F \right \} \\
m_i & := \prod_{p \in S_i} p \;\; (i=1, 2).
\end{split}
\end{equation*}
Then Proposition \ref{prop22b} implies that
$$
C(L) := \frac{D_L}{D_F D_K} = \frac{D_F D_K^3}{c m_1^2m_2^4}
$$
for a positive integer $c \leq \beta := 2^{12} 3^4$ (cf. \cite[Table 3]{MTTW20}) so
\begin{equation} \label{eq43b1}
\begin{split}
N_3^{\tor}(X; H_{48, b})
& = \# \left \{ (F, K) \in \NF_4(S_4) \times \NF_2 : K \cap F^c = \Q \text{ and } C(L) \leq X \right \} \\
& \leq N_{4,2}^{F,K}(\beta X; S_4; \frac{D_F D_K^3}{m_1^2m_2^4}) \\
& =: A_1(\beta X).
\end{split}
\end{equation}
The number $A_1(X)$ can be bounded as in \cite[Section 4.2]{Lee21}.
\begin{itemize}
    \item For given $F$, $m_1$ and $m_2$, the number of $K \in \NF_2$ which satisfies $m_1m_2 \mid D_K$ and $\displaystyle \frac{D_K}{m_1m_2} \leq \left ( \frac{m_2X}{m_1 D_F} \right )^{\frac{1}{3}}$ is at most $\displaystyle 2 \left ( \frac{m_2X}{m_1 D_F} \right )^{\frac{1}{3}}$. 
    
    \item For every $F \in \NF_4(S_4)$ and its cubic resolvent $E \in \NF_3(S_3)$, there is $f \in \Sqf_2$ such that $D_F = D_E f^2$. By \cite[Theorem 5.1]{CT16}, $m_1$ divides $D_E$ and $m_2$ divides $f$.
    
    \item By \cite[Theorem 1.4 and Proposition 6.4]{CT16}, the number of $F \in \NF_4(S_4)$ such that $D_F = D_E f^2$ for given $E \in \NF_3(S_3)$ and $f \in \Sqf_2$ is bounded by $O(h_2(E) \cdot 3^{w(f)})$.
    
    \item If $\displaystyle \frac{m_2X}{m_1 D_F} = \frac{m_2X}{m_1 D_E f^2} \geq 1$, then $\displaystyle D_E f^2 \leq \frac{m_2 X}{m_1} \leq fX$ so $D_Ef^2 \leq X^2$. 
\end{itemize} 
Combining all of these, we obtain
\begin{equation} \label{eq43b2}
A_1(X) \ll \sum_{i=0}^{\left \lfloor 2 \log_2 X \right \rfloor} A_2(X; 2^i)
\end{equation}
for
\begin{equation} \label{eq43b3}
A_2(X; B) := 
\sum_{\substack{m_1,  m_2 \in \, \Sqf \\ (m_1m_2, 6)=1 \\ (m_1, m_2)=1}}
\sum_{\substack{E \in \NF_3(S_3) \\ m_1 \mid D_E}} 
\sum_{\substack{f \in \, \Sqf_2 \\ m_2 \mid f \\ D_E f^2 \leq \frac{m_2 X}{m_1} \\ B \leq D_Ef^2 < 2B}}
\left ( \frac{m_2}{m_1} \cdot \frac{X}{B} \right )^{\frac{1}{3}} h_2(E) \cdot 3^{w(f)}.
\end{equation}
Now we bound the right-hand side of the inequality (\ref{eq43b3}).
\begin{itemize}
    \item The inequality $\displaystyle B \leq D_E f^2 \leq \frac{m_2X}{m_1}$ implies that $\displaystyle m_1 \leq \frac{m_2X}{B}$. Therefore the sum $\displaystyle \sum_{m_1, m_2} \left ( \frac{m_2}{m_1} \right )^{\frac{1}{3}}$ for given $E$ and $f$ is bounded by
    \begin{equation} \label{eq43b4}
    \begin{split}
    \sum_{m_2 \mid f} \sum_{m_1 \leq \frac{m_2X}{B}} \left ( \frac{m_2}{m_1} \right )^{\frac{1}{3}} 
    & \ll \sum_{m_2 \mid f} m_2^{\frac{1}{3}} \left ( \frac{m_2X}{B} \right )^{\frac{2}{3}} \\
    & = \left ( \frac{X}{B} \right )^{\frac{2}{3}} \sum_{m_2 \mid f} m_2 \\
    & \ll \left ( \frac{X}{B} \right )^{\frac{2}{3}} f \log \log X.
    \end{split}
    \end{equation}
    The last inequality is due to the classical upper bound $\sigma (n) \ll n \log \log n$ of Gronwall \cite{Gro13}.
    
    \item The inequality (\ref{eq43b4}) implies that
    \begin{equation} \label{eq43b5}
    \begin{split}
    A_2(X; B) 
    & \ll \frac{X \log \log X}{B}
    \sum_{f < (2B)^{\frac{1}{2}}} 3^{w(f)} f
    \sum_{\substack{E \in \NF_3(S_3) \\ D_E < \frac{2B}{f^2}}} h_2(E) \\
    & \ll \frac{X \log \log X}{B}
    \sum_{f < (2B)^{\frac{1}{2}}} 3^{w(f)} f \cdot \frac{2B}{f^2} \;\; \text{(\cite[Theorem 5]{Bha05})} \\
    & \ll X \log \log X \sum_{f < (2B)^{\frac{1}{2}}} \frac{\tau(f)^c}{f} \;\; (c := \frac{\log 3}{\log 2})  \\
    & \ll X \log \log X (\log B)^3. 
    \end{split}
    \end{equation}
    The last inequality can be deduced using Lemma \ref{lem23c} as in the inequality (\ref{eq43a5}).
\end{itemize}
By the inequalities (\ref{eq43b1}), (\ref{eq43b2}) and (\ref{eq43b5}), we have
$$
N_3^{\tor}(X; H_{48, b}) 
\ll \sum_{i=0}^{\left \lfloor 2 \log_2 (\beta X) \right \rfloor} X \log \log X (\log 2^i)^3
\ll X(\log X)^4 \log \log X.
$$

\item Its proof is similar to the proof of (\ref{43b3}). Following the arguments above, we have
$$
N_3^{\tor}(X; H_{24, b}) 
\ll \sum_{i=0}^{\left \lfloor 2 \log_2 (\beta X) \right \rfloor} A_3(\beta X; 2^i)
$$
for
$$
A_3(X; B) := \frac{X \log \log X}{B}
    \sum_{f < (2B)^{\frac{1}{2}}} 3^{w(f)} f
    \sum_{\substack{E \in \NF_3(C_3) \\ D_E < \frac{2B}{f^2}}} h_2(E).
$$
By \cite[Theorem 1.1]{BSTTTZ20}, $h_2(E) \ll D_E^{\frac{3}{10}}$ so
\begin{equation} \label{eq43b6}
\begin{split}
A_3(X; B)
& \ll \frac{X \log \log X}{B}
    \sum_{f < (2B)^{\frac{1}{2}}} 3^{w(f)} f \left ( \frac{2B}{f^2} \right )^{\frac{4}{5}} \\
& \ll \frac{X \log \log X}{B^{\frac{1}{5}}}
    \sum_{f < (2B)^{\frac{1}{2}}} \frac{\tau(f)^c}{f^{\frac{3}{5}}} \;\; (c = \frac{\log 3}{\log 2}) \\
& \ll X \log \log X (\log B)^2 
\end{split}    
\end{equation}
(following the inequality (\ref{eq43a5})) and
\begin{equation*}
N_3^{\tor}(X; H_{24, b})
\ll \sum_{i=0}^{\left \lfloor 2 \log_2 (\beta X) \right \rfloor}
X \log \log X (\log 2^i)^2
\ll X (\log X)^3 \log \log X. \qedhere
\end{equation*}
\end{enumerate}
\end{proof}

\begin{remark} \label{rmk43c}
Cohen-Martinet heuristics for cyclic cubic fields and $p=2$ \cite[p. 128]{CM87} implies that
$$
\sum_{\substack{E \in \NF_3(C_3) \\ D_E < \frac{2B}{f^2}}} h_2(E)
\ll \frac{B^{\frac{1}{2}}}{f}.
$$
However this does not improve the upper bound of the inequality (\ref{eq43b6}).
\end{remark}

So far, we have proved that 
$$
N_3^{\tor}(X; H) \ll X^{\frac{1}{a(H)}} (\log X)^{O(1)}
$$ 
(under the assumption of Conjecture \ref{conj1d} and Malle's conjecture for quartic $A_4$-fields) if $H$ is not conjugate to $H_{8, g}$ ($H_{8, k}$) or $H_{24, g}$ ($H_{24, e}$, $H_{24, i}$). Unfortunately, our upper bounds for the cases $H=H_{8, g}$ and $H=H_{24, g}$ are much weaker. 

\begin{proposition} \label{prop43d}
$N_3^{\tor}(X; H_{8, g}) \ll X^{\frac{3}{4} - \delta}$ for some $\delta > 0$.
\end{proposition}

\begin{proof}
We follow the notation in \ref{item3ix} of Section \ref{Sec3}. Let $m_L$ be the product of odd primes whose splitting type in $L$ is $(1^41^4)$ or $(2^4)$. By \cite[Table 1]{ASVW21}, $\displaystyle v_p(\frac{D_M}{D_K^2}) = v_p(m_L^2)$ for every odd prime $p$ so
\begin{equation} \label{eq43d1}
   \left ( \frac{D_K D_{M_2}}{D_{K_2}} \right )^2 \overset{\scriptsize (\ref{eq35g})}{=} D_L \frac{D_K^2}{D_M} = \frac{D_L}{2^{\alpha} m_L^2} 
\end{equation}
for $\displaystyle \alpha = v_2(\frac{D_M}{D_K^2}) \leq 6$. 

For an odd squarefree integer $q > 0$, denote by $n(X; q)$ the number of octic $D_4$-fields $L$ such that $D_L \leq X$ and $m_L=q$. By \cite[Corollary 3]{SV21}, we have
\begin{equation} \label{eq43d2}
n(X; q) \ll_{\varepsilon} \frac{X^{\frac{1}{4}} (\log X)^2}{q^{\frac{3}{2} - \varepsilon}} + \frac{X^{\frac{1}{4} - \delta}}{q^{1+\delta}}
\end{equation}
for some constant $\delta > 0$. Now the equation (\ref{eq43d1}) and the inequality (\ref{eq43d2}) imply that
\begin{equation*}
\begin{split}
N_3^{\tor}(X; H_{8, g}) 
& = N_8^{L}(X; D_4; \frac{D_K D_{M_2}}{D_{K_2}}) \\
& \leq N_8^{L}(2^6 X^2; D_4; \frac{D_L}{m_L^2}) \\
& \leq \sum_{q \leq X^{\frac{1}{2}}} n(2^6 X^2 q^2; q) \\
& \ll_{\varepsilon} \sum_{q \leq X^{\frac{1}{2}}} \left ( \frac{X^{\frac{1}{2}} (\log X)^2}{q^{1 - \varepsilon}} + \frac{X^{\frac{1}{2} - 2 \delta}}{q^{\frac{1}{2} + 3 \delta}} \right ) \\
& \ll_{\varepsilon} X^{\frac{1}{2} + \varepsilon} + X^{\frac{3}{4} - \frac{7}{2} \delta}. \qedhere
\end{split}    
\end{equation*}
\end{proof}

\begin{lemma} \label{lem43e}
Let $D_i$ ($i=2, 3, 4, 6, 8, 12$) be as in \ref{item3xvi} of Section \ref{Sec3}. Then for every prime $p \geq 5$, 
\begin{enumerate}
    \item $\displaystyle v_p(\frac{D_8}{D_4 D_2}) \geq v_p(D_3 f)$ for $f \in \Sqf_2$ such that $D_4 = D_3 f^2$.
    
    \item $\displaystyle v_p(\frac{D_{12}}{D_6 D_4}) \geq \frac{2}{9}v_p(D_{12})$.
\end{enumerate}
\end{lemma}

\begin{proof}
\begin{enumerate}
    \item Assume that $\displaystyle v_p(\frac{D_8}{D_4 D_2}) < v_p(D_3 f)$ for a prime $p \geq 5$. 
    \begin{itemize}
        \item Since $\displaystyle v_p(\frac{D_8}{D_4 D_2}) \geq v_p(\frac{D_3f^2}{D_2})$, we have $v_p(D_2)=1$ and $v_p(f)=0$. 
        
        \item $L_2$ is the quadratic resolvent of $L_3$ so $v_p(D_3) \leq 2$ is odd. Therefore $v_p(D_3)=v_p(D_4)=1$.
    
        \item Now $v_p(D_8) < v_p(D_4D_2D_3f) = 4 = v_p(D_2^4)$, which is impossible since $L_2 \subset L_8$.
    \end{itemize}
    
    \item Since $v_p(D_{12}) \geq \max (2v_p(D_6), 3v_p(D_4))$, it is enough to show that at least one of $\displaystyle v_p(D_6) \geq \frac{9}{5}v_p(D_4)$ or $\displaystyle v_p(D_6) \leq \frac{4}{3}v_p(D_4)$ holds. If $v_p(f)=0$, then $v_p(D_6)=v_p(D_4D_3)=2v_p(D_4)$. If $v_p(f)=1$, then $v_p(D_6) = 2v_p(D_4) - 2$. Since $L_4$ is tamely ramified at $p$, $v_p(D_4) \leq 3$ so $\displaystyle v_p(D_6) \leq \frac{4}{3}v_p(D_4)$. \qedhere
\end{enumerate}
\end{proof}

Up to conjugation, there are two transitive subgroups of $S_{12}$ which are isomorphic to $S_4$: $12T8$ and $12T9$. (These groups can be found online at LMFDB \cite{LMFDB}.) For a number field $K$ of degree $12$ which satisfies $\Gal(K^c/\Q) \cong S_4$, the Galois group $\Gal(K^c/\Q)$ is $12T8$ in $S_{12}$ if and only if $K$ has a quartic subfield. Malle's conjecture predicts that
$$
\left\{\begin{matrix}
N_{12}(X; 12T8)  \sim CX^{\frac{1}{5}} \\ 
N_{12}(X; 12T9)  \sim CX^{\frac{1}{4}}.
\end{matrix}\right.
$$

\begin{proposition} \label{prop43f}
\begin{enumerate}
    \item $N_3^{\tor}(X; H_{24, g}) \ll X (\log X)^3$.
    
    \item Under the assumption of Malle's conjecture for $12T8$, we have $N_3^{\tor}(X; H_{24, g}) \ll X^{\frac{9}{10}}$.
\end{enumerate}
\end{proposition}

\begin{proof}
\begin{enumerate}
    \item By Lemma \ref{lem43e}, 
    $$
    N_3^{\tor}(X; H_{24, g}) \leq N_4^{L_4}(\beta X; S_4; D_3f)
    $$
    for some constant $\beta >0$. As in the proof of Proposition \ref{prop43b}, we have
    \begin{equation*}
    \begin{split}
    N_3^{\tor}(X; H_{24, g}) 
    & \ll \sum_{f \leq \beta X} 3^{w(f)} \sum_{\substack{E \in \NF_3(S_3) \\ D_E \leq \frac{\beta X}{f}}} h_2(E)  \\
    & \ll \sum_{f \leq X} \frac{3^{w(f)} \beta X}{f} \\
    & \ll X (\log X)^3 \;\; \text{(Lemma \ref{lem23c})}.
    \end{split}    
    \end{equation*}
    
    \item By Lemma \ref{lem43e}, 
    $$
    N_3^{\tor}(X; H_{24, g}) \leq \# \left \{ L_{12} : D_{12} \leq \beta X^{\frac{9}{2}} \right \}
    $$
    for some constant $\beta >0$. Since $L_{12}$ has a quartic subfield $L_4$, the Galois group $\Gal(L_{12}^c/\Q)$ is $12T8$ so
    \begin{equation*}
    N_3^{\tor}(X; H_{24, g}) 
    \leq N_{12}(\beta X^{\frac{9}{2}} ; 12T8)
    \ll X^{\frac{9}{10}}. \qedhere
    \end{equation*}
\end{enumerate}
\end{proof}

Summing up the results of this section, we obtain the following theorem. Since there is an asymptotic lower bound $X (\log X)^2 \ll N_3^{\tor}(X)$ (\cite[Section 3]{Lee21}), the ratio of the upper and lower bounds of $N_3^{\tor}(X)$ is $O((\log X)^2 \log \log X)$ under the Cohen-Lenstra heuristics for $p=3$. 

\begin{theorem} \label{thm43g}
\begin{enumerate}
    \item We have
    \begin{equation} \label{eq43g1}
    N_3^{\tor}(X) \ll_{\varepsilon} X^{1 + \frac{\log 2 + \varepsilon}{\log \log X}}.
    \end{equation}
    
    \item Under the assumption of Conjecture \ref{conj1d}, we have
    \begin{equation} \label{eq43g2}
    N_3^{\tor}(X) \ll X (\log X)^4 \log \log X.
    \end{equation}
\end{enumerate}
\end{theorem}

\section*{Acknowledgments}

The author is supported by a KIAS Individual Grant (MG079601) at Korea Institute for Advanced Study. 
We thank Ila Varma for sharing the preprint \cite{SV21}. We also thank Frank Thorne and Chia-Fu Yu for their helpful comments.


\end{document}